\documentclass[a4paper,11pt]{amsart}
\usepackage{tikz}
\usepackage{amsfonts,amsmath,amssymb}
\usepackage{color}
\usepackage{algorithmic}
\usepackage{algorithm}
\usepackage{comment}
\usepackage{graphics}
\usepackage[hidelinks]{hyperref}
\usepackage{pgfplots,pgfplotstable}
\pgfplotsset{compat=1.18}
\usepackage{standalone}
\usepgfplotslibrary{groupplots}

\usepackage{xargs}
\usepackage{xifthen}

\usepackage{listings}
\usepackage{mathtools}
\usepackage{esint}
\newtheorem{theorem}{Theorem}

\theoremstyle{definition}

\theoremstyle{lemma}
\newtheorem{lemma}[theorem]{Lemma}

\theoremstyle{remark}
\newtheorem{remark}[theorem]{Remark}

\newtheorem{assumption}[theorem]{Assumption}
\numberwithin{theorem}{section}
\numberwithin{equation}{section}
\numberwithin{table}{section}
\numberwithin{figure}{section}

\usepackage{caption}
\usepackage{subcaption}
\usepackage{graphics}

\definecolor{aliceblue}{rgb}{0.94, 0.97, 1.0}

\newcommand{\ddiv}{\operatorname{div}}
\newcommand{\calT}{\mathcal{T}}
\newcommand{\R}{\mathbb{R}}
\newcommand{\N}{\mathbb{N}}

\newcommand{\opnorm}[1]{{\big\vert\kern-0.25ex\big\vert\kern-0.25ex\big\vert #1 \big\vert\kern-0.25ex\big\vert\kern-0.25ex\righ\vert}}

\newcommand{\calC}{\mathcal{C}}

\newcommand\dx{\,\mathrm{d}x}

\newcommand\ds{\,\mathrm{d}s}

\newcommand{\frakB}{\mathfrak{B}}
\newcommand{\frakD}{\mathfrak{D}}
\newcommand{\frakE}{\mathfrak{E}}

\newcommand{\V}{H^1_0(\Omega)}
\newcommand{\Vp}{H^{-1}(\Omega)}
\newcommand{\LL}{L^2(\Omega)}

\newcommand\calP{\mathcal{P}}

\newcommand\calE{\mathcal{E}}
\newcommand\calB{\mathcal{B}}
\newcommand{\calN}{\mathcal{N}}
\newcommand{\eps}{\varepsilon}

\renewcommand{\TH}{\mathcal{T}_H}
\newcommand{\VH}{V_H}

\newcommand{\tUH}{\tilde{U}_H}
\newcommand{\tUl}{\tilde{U}_H^{\ell}}
\newcommand{\PiH}{\Pi_H}
\newcommand{\tPiH}{\tilde{\calP}}

\newcommand{\tuH}[1]{\tilde{u}_H^{{#1}}}
\newcommand{\tul}[1]{\tilde{u}_H^{\ell,{#1}}}

\newcommand{\eH}[1]{e_H^\ell(t_{{#1}})}

\newcommand{\tel}[1]{\tilde{e}_H^{\ell}(t_{#1})}

\newcommand{\tzH}[1]{\tilde{z}_H(t_{#1})}
\newcommand{\tzl}[1]{\tilde{z}_H^\ell(t_{#1})}
\newcommand{\vH}{v_H}
\newcommand{\tvH}{\tilde{v}_H}
\newcommand{\tvl}{\tilde{v}_H^{\ell}}

\newcommandx*{\ScalarProduct}[3][1=]{#1(#2,#3)}
\newcommandx*{\SP}[2]{\ScalarProduct{#1}{#2}}
\newcommandx*{\IP}[2]{\ScalarProduct[a]{#1}{#2}}

\newcommand{\Cl}{\calC^{\ell}}
\newcommand{\CCl}{(\calC-\Cl)}
\newcommand{\tB}{\tilde{\frakB}}

\newcommand{\halfpos}{n+\frac{1}{2}}
\newcommand{\halfneg}{n-\frac{1}{2}}
\newcommand{\Dt}{D_\tau}

\newcommand{\supp}{\mathrm{supp}}

\newcommand{\CFL}{\Con[b]\Con[inv]\sqrt{(\tfrac{1}{4}-\theta)\beta (1+\Cellsq)}}
\newcommand{\Cell}{\Con[loc]\ell^{\frac{d}{2}}\exp(-\Con[d]\ell)}
\newcommand{\Cellsq}{\Con[loc]^2\ell^{d}\exp(-2\Con[d]\ell)}

\newcommandx*{\Con}[1][1]{C_{\mathrm{#1}}}
\newcommandx*{\con}[1][1]{c_{#1}}

\newcommand*{\Nb}[2]{{\mathtt{N}^{#1}}(#2)}

\newcommand{\splod}{$p$-LOD-$\theta$}

\definecolor{darkgreen}{rgb}{0.09, 0.45, 0.27}
\definecolor{myOrange}{rgb}{0.85000,0.32500,0.09800}

\begin{document}
%
%
\title[High-order multiscale method for waves]{
A Higher-Order Multiscale Method for the Wave Equation}
\author[F.~Krumbiegel, R.~Maier]{Felix Krumbiegel$^\dagger$ and Roland Maier$^\dagger$}
\address{${}^{\dagger}$ Institute for Applied and Numerical Mathematics, Karlsruhe Institute of Technology, Englerstr.~2, 76131 Karlsruhe, Germany}
\email{\{felix.krumbiegel,roland.maier\}@kit.edu}
\date{\today}
%
%
\begin{abstract}
In this paper we propose a multiscale method for the acoustic wave equation in highly oscillatory media. We use a higher-order extension of the localized orthogonal decomposition method combined with a higher-order time stepping scheme and present rigorous a-priori error estimates in the energy-induced norm. We find that in the very general setting without additional assumptions on the coefficient beyond boundedness, arbitrary orders of convergence cannot be expected but that increasing the polynomial degree may still considerably reduce the size of the error. Under additional regularity assumptions, higher orders can be obtained as well. Numerical examples are presented that confirm the theoretical results.
\end{abstract}

\maketitle


%
%
\section{Introduction}\label{sec:intro}
Modeling wave propagation through heterogeneous media typically requires a discretization scale that resolves the heterogeneities that are encoded by an oscillatory coefficient in the underlying partial differential equation (PDE). 
This is particularly challenging if the coefficient varies on a very fine scale~$0<\varepsilon\ll 1$.
It is well-known that standard approaches such as the finite element method (FEM) need to be defined on a mesh with mesh size~$h<\varepsilon$ to obtain reasonable approximation properties of the corresponding discrete solution, even in a macroscopic sense.
To avoid expensive computations on rather fine scales, especially in a time-dependent scenario where a system of equations has to be solved in every time step, specifically designed multiscale methods present an alternative.

Respective approaches for the wave equation are, e.g., presented in~\cite{AbdG11, EngHR11} based on the heterogeneous multiscale method~\cite{EE03, EE05, AbdEEV12} or in~\cite{OwhZ17} using so-called gamblets~\cite{Owh17, OwhS19}. Further methods are presented in~\cite{OwhZ08} and~\cite{JiaEG10, JiaE12} based on suitable global coordinate transformations. Operator-based upscaling for the wave equation is considered in~\cite{VdoMK05,KorM06}.
The localized orthogonal decomposition (LOD) method is another prominent multiscale approach and constructs coarse-scale spaces with built-in information on the coefficient on finer scales. The method is originally stated in an elliptic setting~\cite{HenP13,MalP14,MalP20} and based on the ideas of the variational multiscale method~\cite{Hug95, HugFMQ98}. It provably works under minimal structural assumptions on the coefficient and allows computations to be localized to patches such that global fine-scale computations can be completely avoided.
Previously, the LOD method has been applied to the wave equation combined with, e.g., a Crank--Nicolson scheme in~\cite{AbdH17}, a leapfrog scheme in~\cite{PetS17,MaiP19}, or a mass-lumped leapfrog scheme in~\cite{GeeM23} for the temporal discretization. The approach has also been used in the context of a damped wave equation in~\cite{LjuMP21} or in connection with waves on spatial networks~\cite{GorLM23}. For the treatment of space- and time-dependent coefficients with slow variations in time, suitable updates of the spatial discretization are required as studied in~\cite{MaiV22}. We refer to the survey book chapter~\cite{AbdH17b} for further details and methods for the heterogeneous wave equation.

In this work, we focus on a multiscale approach in the spirit of the LOD method for the spatial discretization. Aiming for higher-order convergence rates in the multiscale setting, we apply a higher-order variant of the LOD method based on piecewise polynomials up to degree $p$, which we refer to as the {$p$-LOD} method. The method has been proposed in~\cite{Mai20,Mai21} based on a gamblet formulation, and was further refined in~\cite{DonHM23}.
When aiming for higher orders in time, there is the need for a suitable time discretization scheme as common methods such as the leapfrog or Crank--Nicolson scheme only achieve second-order rates. Possible methods include the $\theta$-scheme used in \cite{Kar11}, see also the previous works~\cite{LimKD07,KimL07}, which is a special case of the Newmark scheme~\cite{New59}, also referred to as {Newmark-$\beta$ method}. With an appropriate choice of $\theta$, convergence rates up to fourth order in time can be achieved.
Other possibilities involve rewriting the wave equation into a first order formulation and applying higher-order Runge-Kutta schemes as used, e.g., in~\cite{GroMM15,AlmM17}.

Here, we combine the $p$-LOD method with the $\theta$-scheme as used in~\cite{Kar11}. 
Interestingly, arbitrarily high convergence rates in space as observed for the elliptic setting with minimal assumptions in~\cite{Mai21,DonHM23} can only be obtained for smooth coefficients in the context of the wave equation, which is also to be expected from our theoretical investigations. Nevertheless, increasing the polynomial degree turns out to have a positive effect on the size of the error with respect to the number of degrees of freedom, even without additional smoothness assumptions.

The remainder of this paper is structured as follows. In Section~\ref{sec:method}, we introduce the general setting and the considered multiscale method regarding the spatial and the temporal discretization. In Section~\ref{sec:error}, stability and error estimates are derived, and numerical examples are presented in Section~\ref{sec:examples}. In particular, we compare our approach with the classical first-order LOD method for the wave equation, as, e.g., considered in~\cite{AbdH17}.

%
%
\section{A higher-order multiscale method}\label{sec:method}
\subsection{Model problem}\label{subsec:model}
We consider the wave equation with homogeneous Dirichlet boundary conditions
\begin{equation}\label{eq:model}
\begin{aligned}
\ddot{u}-\ddiv(A\nabla u)=&\ f\qquad &&\text{in~} (0,T)\times\Omega,\\
u(0)=&\ u_0\qquad &&\text{in~}\Omega,\\
\dot{u}(0)=&\ v_0\qquad &&\text{in~}\Omega,\\
u\vert_{\partial\Omega}=&\ 0\qquad &&\text{in~}(0,T),
\end{aligned}
\end{equation}
where $\Omega\subset\R^d$ is a polygonal, convex, bounded Lipschitz domain for $d \in\{1,2,3\}$. We assume that~$u_0\in \V$, $v_0\in\LL$, $f\in L^2(0,T;\LL)$ and $A\in L^{\infty}(\Omega)$ with $\alpha \leq A(x)\leq\beta$ for almost all $x\in\Omega$, where~$0<\alpha\leq \beta<\infty$. Specifically, we have coefficients in mind that can vary on a fine scale $0<\eps\ll 1$ but a precise dependence on $\eps$ is not required. 
Note that we might as well consider matrix-valued coefficients and more involved boundary conditions, but we restrict ourselves to the case of scalar coefficients and zero Dirichlet boundary conditions to simplify the presentation.

The variational formulation of equation~\eqref{eq:model} reads as follows: we seek a solution $u\in L^2(0,T;\V)$ with $\dot{u}\in L^2(0,T;\LL)$ and $\ddot{u}\in L^2(0,T;\Vp)$ such that
\begin{equation}\label{eq:var_form}
\langle\ddot{u},v\rangle_{\Vp\times \V} +a(u,v)=( f,v)_{\LL}
\end{equation}
for all $v\in\V$ with $u(0)=u_0$ and $\dot{u}(0)=v_0$ and with the bilinear form $a(u,v)\coloneqq(A\nabla u,\nabla v)_{\LL}$. Under the above assumptions on the data, the variational formulation has a unique solution, see, e.g., \cite[Ch.~3]{LioM72}.

\subsection{Discretization in space}\label{subsec:space_discr}
As mentioned above, classical spatial discretizations based on, e.g., the finite element method need to resolve fine-scale features of the coefficient in order to achieve meaningful approximations in the first place. Otherwise, pre-asymptotic effects may be observed; see also the numerical experiments in~\cite{GeeM23}. To obtain reasonable approximations with a moderate amount of degrees of freedom, more involved constructions are necessary.
Here, we construct a problem-adapted finite-dimensional subspace of~$H^1_0(\Omega)$ following the ideas of~\cite{DonHM23} based on the earlier work~\cite{Mai21} regarding higher-order multiscale approaches in the spirit of the LOD. 
First, a finite-dimensional polynomial space is set up, which is altered to obtain a conforming construction and lastly the space is tailored to the problem at hand.

Let $\{\TH\}_{H>0}$ be a family of regular decompositions of $\Omega$ into quasi-uniform $d$-rectangles; cf.~\cite[Chs. 2 \& 3]{Cia78}. For a given mesh size $H > 0$ and a fixed polynomial degree $p\in\N_0$, we define $\VH$ as the discontinuous and $\hat{V}_H$ as the continuous spaces of piecewise polynomials up to partial degree $p$ with respect to the mesh~$\TH$, i.e.,
\begin{equation*}
	\begin{aligned}
		\VH&\coloneqq\{v\in\LL\mid\forall K\in\TH:v\vert_K\text{ is a polynomial of partial degree}\leq p\},\\
		\hat{V}_H&\coloneqq\{v\in C^0(\Omega)\mid\forall K\in\TH:v\vert_K\text{ is a polynomial of partial degree}\leq p\}.
	\end{aligned}
\end{equation*}
Note that also $d$-parallelograms could be considered. In that case the definition of a basis of $\VH$ would differ and be done using a reference element with a linear map to each parallelogram of the mesh. In this case, the constants in the estimates possibly change with the shape of the parallelogram. For simplicity, we restricted ourselves to $d$-rectangles. To adjust the current setting accordingly, the results in~\cite[Lem.~3.4, Thm.~3.5, and Cor.~3.6]{Mai21} need to be transferred to $d$-parallelograms, which is possible using~\cite{Geo08}. 

For any subdomain $S\subset\Omega$, we define the restricted spaces
\begin{equation*}
\VH(S)=\{v\in\VH\mid \supp(v)\subset S\}, \qquad
\hat{V}_H(S)=\{v\in\hat{V}_H\mid \supp(v)\subset S\},
\end{equation*}
and denote by $\PiH\colon\LL\to\VH$ the $L^2$-projection onto $\VH$. That is, $\PiH$  is defined element-wise for~$v\in\LL$ by
\begin{equation}\label{eq:Pi_definition}
\big((\PiH v)\vert_K,w\big)_{L^2(K)}=(v\vert_K,w)_{L^2(K)}
\end{equation}
for all $w\in\VH(K)$ and all $K\in\TH$. Choosing $w=(\PiH v)\vert_K$ as a test function in \eqref{eq:Pi_definition}, we directly obtain the stability estimate 
\begin{equation}\label{eq:Pi_stability}
\|\PiH v\|_{L^2(K)}\leq \|v\|_{L^2(K)},
\end{equation}
for all $v\in L^2(K)$. Furthermore, there exists a constant $\Con[pr]>0$ independent of $H$ (but dependent on $p$) such that 
\begin{equation}\label{eq:Cpr}
\|(1-\PiH)v\|_{L^2(K)}\leq\Con[pr] H^k| v|_{H^k(K)};
\end{equation}
for all $v\in H^k(K)$ with $k\leq p+1$, where $1$ denotes the identity operator on $H^1(K)$; see, e.g., \cite{Sch98, HouSS02, Geo03}. Note that the constant $\Con[pr]$ scales like $\big(\frac{(p+1-k)!}{(p+1+k)!}\big)^{\frac{1}{2}}$, and specifically for~$k=1$ we have $\Con[pr]\lesssim (p+1)^{-1}$. 
There also exists a constant $\Con[inv]>0$ dependent on~$p$ such that the inverse estimate
\begin{equation}\label{eq:Cinv}
	\|\nabla v_H\|_{L^2(K)}\leq \Con[inv]H^{-1}\|v_H\|_{L^2(K)}
\end{equation}
holds for $v_H\in \VH$.
Summing up the inequalities \eqref{eq:Pi_stability}--\eqref{eq:Cinv}, respectively, over all elements of the mesh, each inequality can be generalized to the whole domain~$\Omega$ with the use of element-wise gradients on the right-hand sides.

We denote with $\frakD=\bigcup_{K\in\TH}\{\Lambda_{K,j}\}_{j=1}^M$ for $M=(p+1)^d$ the basis of $\VH$ consisting of tensor-product shifted Legendre polynomials on the element $K$. That is, $\Lambda_{K,j}$ is of form $\prod_{k=1}^d\sqrt{2p_k+1}\,L_{p_k}(2x_k-1)$, where~$L_{p_k}$ is the one-dimensional Legendre polynomial of degree $p_k\leq p$ defined on $[-1,1]$. While the non-conformity of the space~$\VH$ is very beneficial in terms of locality, it complicates the construction of a multiscale space in the spirit of the LOD method. To circumvent the issue while still being able to make use of the locality of $\VH$, we instead use for each polynomial a certain (local)~$H^1_0(K)$-conforming bubble function, whose $L^2$-projection coincides with the polynomial itself. \cite[Corollary 3.6]{Mai21} states the existence of functions $\frakB'=\bigcup_{K\in\TH}\{b_{K,j}\}_{j=1}^M$ such that 
\begin{equation}\label{eq:bubble_proj}
b_{K,j}\in H^1_0(K)\quad\textrm{and}\quad\PiH b_{K,j}=\Lambda_{K,j}
\end{equation}
for each $K\in\TH$ and $j=1,\dots,M$. The space $\operatorname{span}(\frakB')$ is conforming and has the same dimension as $\VH$. In the top row of Figure~\ref{fig:bubble}, three constant Legendre polynomials~$\Lambda_{K,1}$ are depicted (left) with their corresponding bubble functions $b_{K,1}$ (right). Note that the depicted bubble functions fulfill the condition~\eqref{eq:bubble_proj} only for polynomial degrees~$p=0,1$. For higher polynomial degrees, the bubble functions need to be suitably adjusted, see~\cite[Rem.~7.1]{DonHM23}. The remark also addressed the practical calculation of the bubble functions.

\begin{figure}
\centering
\includegraphics[scale=.35]{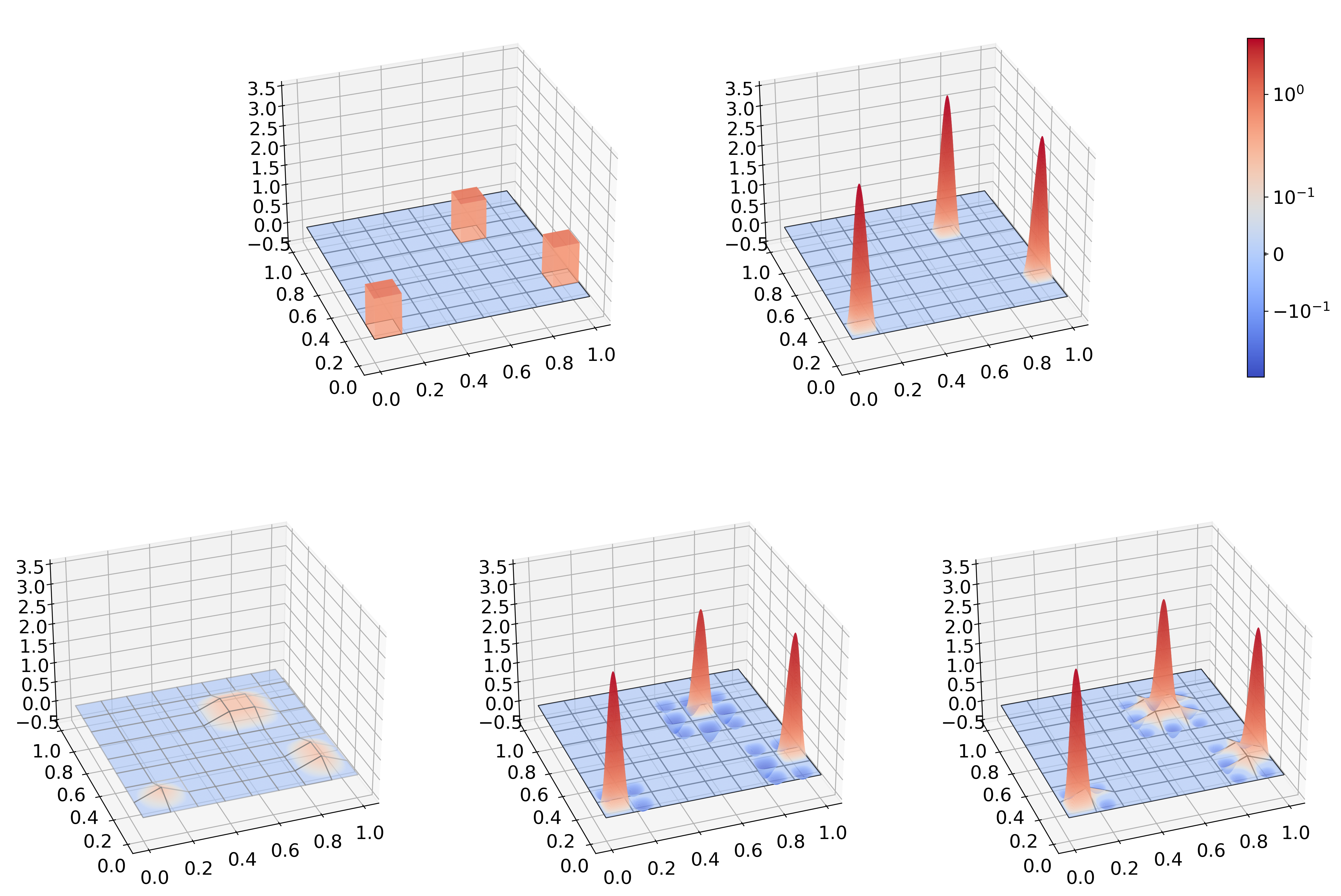}
\caption{Plots of different basis functions: three Legendre polynomials $\Lambda_{K,1}$ (top left), the corresponding bubble functions $b_{K,1}$ (top right), the functions $\iota_K$ (bottom left), $\nu_K$ (bottom middle), and the newly created basis function $\iota_K+\nu_K$ (bottom right).}\label{fig:bubble}
\end{figure}

It turns out that the (true) locality of the bubble functions $b_{K,1}$, which approximate the constant Legendre polynomials $\Lambda_{K,1}$ cause problems when constructing a localized multiscale method. Therefore, we suitably extend the support of the corresponding bubble function by one layer of elements without changing its $L^ 2$-projection. Note that this is a rather technical detail which is discussed and justified in~\cite{DonHM23} based on the earlier works~\cite{AltHP21,HauP22}. 
Here we present a practical construction of such extended bubble functions and the corresponding multiscale spaces. 

First, we define for a subdomain $S$ the patch of size $k\in\N$  recursively as the set~$\Nb{k}{S}=\Nb{1}{\Nb{k-1}{S}}$, $k>1$ with $\Nb{1}{S}\coloneqq\{K\in\TH\mid \overline K\cap \overline S\neq\emptyset\}$.
Let~$\iota_K$ for $K \in \TH$ be such that
\begin{equation*}
\iota_K\in \hat{V}_H^{1}(\Nb{1}{K}),\quad\iota_K(\hat{x})=\begin{cases}
\frac{1}{2^d},\quad& \hat{x}\in \overline K\setminus\partial\,\Nb{1}{K},\\
0,\quad &\hat{x}\in\partial\,\Nb{1}{K},\end{cases}
\end{equation*}
where $\hat{V}_H^1$ refers to the space $\hat{V}_H$ with the specific choice $p=1$ and $\hat{x}$ denotes a vertex of the mesh $\TH$.
Furthermore, for each element $K\in\TH$ we define $\nu_K$ by
\begin{equation*}
\nu_K=\sum_{G\in\TH:\,G\subset \Nb{1}{K} }\sum_{j=1}^{M}c_{K,G,j}\,b_{G,j},\quad c_{K,G,j}=\int\limits_{G}\big(b_{K,1}-\iota_K\vert_{G}\big)\Lambda_{G,j}\dx.
\end{equation*}
Further, we define the set of functions
\begin{equation*}
\frakB=\bigcup_{K\in\TH}\big(\{\iota_K+\nu_K\}\cup\{b_{K,j}\}_{j=2}^M\big),
\end{equation*}
and the corresponding linear operator $\calB_H:\VH\to\calB_H\VH\eqqcolon U_H$ which is determined by
\begin{equation*}
\calB_H\Lambda_{K,j}=\begin{cases}
\iota_K+\nu_K,\quad &j=1,\\
b_{K,j},\quad &j=2,\dots,M.
\end{cases}
\end{equation*}
Note that $\frakB$ is a basis of $U_H$.
Furthermore, it holds $\PiH(\iota_K+\nu_K)=\Lambda_{K,1}$. We emphasize that $\operatorname{span} \frakB \neq \operatorname{span} \frakB'$ but $\PiH (\operatorname{span} \frakB) = \PiH(\operatorname{span} \frakB') = \VH$, which is a necessary property for the construction below. The plots in the bottom row of \ref{fig:bubble} show examples of the newly created basis functions~$\iota_K+\nu_K$ (right) with their respective parts $ \iota_K$ (left) and~$\nu_K$ (middle) for the Legendre polynomial and its bubble of the top row.
Roughly speaking, the new function $\iota_K+\nu_K$ is composed of a piecewise bilinear continuous part on~$\Nb{1}{K}$, which is then adapted by multiple bubbles to fulfill the equality $\PiH(\iota_K+\nu_K)=\Lambda_{K,1}$.

\begin{remark}\label{rem:Cb}
The construction of the operator $\calB_H$ resembles the construction of~$\calP_H$ in~\cite{DonHM23} in the sense of $\calB_H\circ \Pi_H=\calP_H$, such that in particular $\calB_H v_H=\calP_Hv_H$ holds for $v_H\in V_H$, and therefore we get from~\cite[eq.~(3.7)]{DonHM23}, that there exists a constant $\Con[b]>0$, which scales with $p$ like $\Con[b]\lesssim (p+1)^2$, such that
\begin{equation}\label{eq:Cb}
\|\nabla\calB_H\PiH v\|_{L^2(K)}\leq\Con[b] \|\nabla\PiH  v\|_{L^2(\Nb{1}{K})}
\end{equation}
for $v\in L^2(\Nb{1}{K})$ and each element $K\in\TH$.
\end{remark}

Our next step is to define a localized operator $\Cl$ as in \cite[Sec. 5.2]{DonHM23}, that corrects basis functions in a way that the resulting basis has improved approximation properties regarding the underlying multiscale problem. Let $K\in\TH$, $\ell\in\N$, and define the kernel space $W(\Nb{\ell}{K})\coloneqq\ker\PiH\cap H^1_0(\Nb{\ell}{K})$. The \emph{element correction} $\Cl_K\colon\V\to W(\Nb{\ell}{K})$ for $v\in\V$ is then given as the solution~$\Cl_Kv\in W(\Nb{\ell}{K})$ to
\begin{equation}\label{eq:Cl_definition}
a(\Cl_K v,w)=\ a\vert_K(v,w)
\end{equation}
for all $w\in W(\Nb{\ell}{K})$, where $a\vert_K(v, w)\coloneqq\int_KA\nabla v\cdot\nabla w\dx$ denotes the inner product on~$\V$ restricted to the element $K$. Note that the left-hand side in equation \eqref{eq:Cl_definition} is by definition of the spaces restricted to the patch $\Nb{\ell}{K}$. The \emph{correction operator} is defined as the sum of its local parts, i.e., $\Cl=\sum_{K\in\TH}\Cl_K$. If we choose $\ell\in\N$ sufficiently large, such that for each $K\in\calT_H$ the left-hand side of \eqref{eq:Cl_definition} is global, we formally set $\ell=\infty$ and abbreviate $\calC=\calC^\infty$. The choice $\ell=\infty$ is optimal in the sense of convergence properties but in practice this operator is cumbersome to deal with due to globally defined element corrections. In practical calculations, it suffices to use the operator $\Cl$ for moderate choices of $\ell$ instead of $\calC$ due to a localization result with exponentially decreasing errors as stated in the following lemma.
\begin{lemma}\textnormal{(\cite[Lem.~5.2]{DonHM23})}\label{lem:CminCl}
There exist constants $\Con[loc],\Con[d]>0$ independent of $H$ and $\ell$, but dependent on $p$ such that for all $v\in\V$ and $\ell\in\N$ it holds 
\begin{equation}\label{eq:CminCl}
\|\nabla(\calC-\Cl)v\|_{\LL}\leq\Con[loc]\ell^{\frac{d}{2}}\exp(-\Con[d]\ell)\|\nabla v\|_{\LL}.
\end{equation}
\end{lemma}

We are now set to define an appropriate multiscale space. We define
\begin{equation*}
\tilde{b}^\ell_{K, j}=\begin{cases}
(1-\Cl)(\iota_K+\nu_K),\quad &j=1,\\
(1-\Cl)b_{K,j},\quad &j=2,\dots,M,
\end{cases}
\end{equation*}
and the multiscale space $\tUl=\operatorname{span}\tB^\ell$ given by the basis $\tB^\ell=\bigcup_{K\in\TH}\{\tilde{b}^\ell_{K,j}\}_{j=1}^M$. The basis functions in $\tB^\ell$ are locally supported on $ \Nb{\ell}{K}$ for some $K\in\TH$ (except for $\tilde{b}^\ell_{K,1}$ with support on $\Nb{\ell+1}{K}$), and it holds~$\Pi_H\tUl=\VH$ with both spaces having the same dimension. The new basis functions include a correction by the problem-dependent operator $\Cl$ such that they achieve better approximation properties for the underlying problem. 
Analogously to above, we set $\tUH=\tilde{U}^\infty_H$ for $\ell=\infty$. Note that the construction of the basis~$\tB^\ell$ resembles the construction in \cite[Sec.~6]{DonHM23}, but the presentation here follows a more practical approach. We emphasize that the basis functions $\tilde{b}_{K,j}$ for $j\geq 2$ can be computed from $\Lambda_{K,j}$ directly; see~\cite{DonHM23}. Thus, only the extended bubble functions~$\iota_K+\nu_K$ need actually to be calculated. With the multiscale space~$\tUl$ we are now ready to state our method in the next subsection. 
First, however, we require some further definitions. We introduce for a series of functions $\{v^n\}_{n=0}^N$ with $v^n\in H^1(\Omega)$ for~$n=0,\dots,N$ the average of two consecutive functions by
\begin{equation*}
	v^{n+\frac{1}{2}}\coloneqq\tfrac{v^{n+1}+v^{n}}{2},
\end{equation*}
and the first and second discrete time derivatives for such a series and a time step $\tau$ by
\begin{equation*}
	\begin{aligned}
		\Dt v^{\halfpos}&\coloneqq\tfrac{v^{n+1}-v^n}{\tau},\quad
		\Dt^2v^{n}&\coloneqq\tfrac{v^{n+1}-2v^n+v^{n-1}}{\tau^2}.
	\end{aligned}
\end{equation*}
Moreover, we use $a\lesssim b$ if $a\leq Cb$ with a generic constant $C>0$.

\subsection{$p$-LOD-$\theta$ method}\label{subsec:splodth}
The multiscale space $\tUl$ introduced in the previous subsection can now be used as trial and test space for the variational formulation~\eqref{eq:var_form} and combined with the $\theta$-scheme as used in~\cite{Kar11} to obtain the \emph{$p$-LOD-$\theta$} method: 
find $\mathbf{\tilde{u}_H^\ell}=\{\tul{n}\}_{n=0}^{N}$ with $\tul{n}\in\tUl$ for $n=0,\dots,N$ such that for $n\geq 1$
\begin{equation}\label{eq:splodth}
\tau^{-2}\big(\tul{n+1}-2\tul{n}+\tul{n-1},\tvH\big)_{\LL}+a\big(\tul{n;\theta},\tvH\big)=\big(f^{n;\theta},\tvH\big)_{\LL}
\end{equation}
for all $\tvH\in\tUl$ and appropriately chosen $\tul{0},\tul{1}\in\tUl$, where $\tau=\frac{T}{N}$, $\theta\in[0,\frac{1}{2}]$ is fixed, and the weighted $\theta$-difference is defined by
\begin{equation}\label{eq:theta_difference}
\tul{n;\theta}=\theta\tul{n+1}+(1-2\theta)\tul{n}+\theta\tul{n-1}.
\end{equation}
For the right-hand side term $f$, we have to suitably choose $f^{n;\theta}$. If it is possible to evaluate~$f$ pointwise, we define $f^{n;\theta}$ analogously to \eqref{eq:theta_difference} where $f^n=f(n\tau)$. 
The time discretization is a generalization of multiple schemes, e.g., for $\theta=0$ we obtain the leapfrog scheme and for~$\theta=\tfrac{1}{4}$ the Crank--Nicolson scheme.

In the following, we show that the error of the $p$-LOD-$\theta$ method applied to the wave equation can be estimated by
\begin{equation*}
\max_{0\leq n<N}\big\|\nabla\big(\tul{\halfpos}-u(t_{\halfpos})\big)\big\|_{\LL}\lesssim H^r+\tau^s,
\end{equation*}
where the spatial convergence rate $r$ depends on the regularity of the initial conditions~$u_0,v_0$, the right-hand side $f$, and potentially even the coefficient $A$. The temporal convergence rate $s$ depends on the choice of $\theta$. For rough heterogeneous coefficients in the elliptic setting, a spatial discretization with the $p$-LOD method leads a convergence rate of order $r=p+2$ if the right-hand side is smooth enough, but interestingly this rate cannot be transferred to the wave equation for very rough coefficients, even for smooth data~($u_0,v_0,f$). Instead, we find that the order is capped at~$r\leq 2$ in the very general setting. The reason is given in Remark~\ref{rem:jk} below based on the theory in Section~\ref{sec:error}, which is also backed by the numerical examples in Section \ref{sec:examples}. 
The order $s$ of the temporal error depends on the choice of $\theta$, we obtain order $s=2$ if $\theta \neq \tfrac{1}{12}$ and order $s=4$ if $\theta=\frac{1}{12}$.

%
%
\section{Stability and error analysis}\label{sec:error}

The proof of an error estimate for the \splod~method is based on energy estimates as in~\cite{Chr09, Jol03}. We split the total error of the method into three parts, namely the temporal error, the localization error and the spatial error, which are treated separately. 
We estimate the errors pointwise in the energy-induced discrete norm $\|D_\tau\cdot\|_{\LL}+\|\nabla\cdot\|_{\LL}$ but emphasize that discrete $L^\infty(0,T;L^2(\Omega))$-error bounds are also possible (cf.~Remark~\ref{rem:L2}) with similar techniques as used in~\cite{Kar11}. 

\subsection{Energy estimation}\label{subsec:nrg}
In order to show stability of the $p$-LOD-$\theta$ method, we first state the useful property that the discrete energy
\begin{equation}\label{eq:nrg_definition}
\begin{aligned}
\calE^{\halfpos}
\coloneqq \frac{1}{2}&\Big[\big\|\Dt\tul{\halfpos}\big\|_{\LL}^2+a\big(\tul{\halfpos},\tul{\halfpos}\big)
+\tau^2(\theta-\tfrac{1}{4}) a\big(\Dt\tul{\halfpos},\Dt\tul{\halfpos}\big)\Big]
\end{aligned}
\end{equation}
is conserved by our numerical scheme.

\begin{lemma}[Energy conservation]\label{lem:nrg}
The \splod~method~\eqref{eq:splodth} conserves the discrete energy~\eqref{eq:nrg_definition} in the sense that
\begin{equation}\label{eq:nrg_estimate}
\big(f^{n;\theta},\tul{n+1}-\tul{n-1}\big)_{\LL}=2\big(\calE^{\halfpos}-\calE^{\halfneg}\big).
\end{equation}
In particular, for a source term $f\equiv 0$ we obtain energy conservation in the classical sense,
\begin{equation*}
\calE^{\halfpos}=\calE^{\halfneg}=\calE^{\frac{1}{2}}.
\end{equation*}
\end{lemma}

\begin{proof}
We choose $\tvH=\tul{n+1}-\tul{n-1}$ as a test function in \eqref{eq:splodth} and obtain~\eqref{eq:nrg_estimate} with straight-forward calculations.
\end{proof}

\subsection{Stability}\label{subsec:stability}
With the energy property~\eqref{eq:nrg_estimate}, it suffices to show that the discrete energy of the $p$-LOD-$\theta$ method is positive at all times to obtain stability. Dependent on the choice of $\theta$, we get an unconditionally stable method for $\theta\geq \frac{1}{4}$ and require a CFL condition to ensure stability for $\theta < \frac{1}{4}$. 

\begin{theorem}[Stability]\label{thm:stability}
For~$\tfrac{1}{4}\leq\theta\leq\tfrac{1}{2}$, the $p$-LOD-$\theta$ method \eqref{eq:splodth} is unconditionally stable.
For~$0\leq\theta<\tfrac{1}{4}$, the $p$-LOD-$\theta$ method \eqref{eq:splodth} is stable if the CFL condition
\begin{equation}\label{eq:CFL}
	\tau \leq \frac{1-\delta}{\Con[CFL]} H,\quad \Con[CFL]= \CFL
\end{equation}
holds for some $\delta>0$, where $\beta$ is the upper bound of $A$, $\Con[loc]$ and $\Con[d]$ are defined in Lemma~\ref{lem:CminCl}, $\Con[b]$ is defined in \eqref{eq:Cb}, and $\Con[inv]$ in \eqref{eq:Cinv}. In the unconditionally stable case, we formally set $\delta = 1$. In both cases, there exists a constant $\Con[s]>0$ dependent on $\delta$, $\alpha$, and $\beta$, such that
\begin{equation}\label{eq:stability_estimate}
\begin{aligned}
\big\| \Dt\tul{\halfpos}\big\|_{\LL}+\big\|\nabla&\tul{\halfpos}\big\|_{\LL}\\
&\leq  \Con[s]\Big(\big\| \Dt\tul{\frac{1}{2}}\big\|_{\LL}+\sqrt{\|\tul{1}\|_{\LL}\|\tul{0}\|_{\LL}}\\
&\qquad\qquad+\tau\sqrt{\theta}\big\|\nabla \Dt\tul{\frac{1}{2}}\big\|_{\LL} +\sum_{k=1}^n\tau\big\|f^{k;\theta}\big\|_{\LL}\Big).
\end{aligned}
\end{equation}
\end{theorem}

\begin{proof}
It is easy to see that $\calE^{n+\frac{1}{2}}\geq 0$ for $\theta\geq\frac{1}{4}$. Now, let $\theta<\frac{1}{4}$. Employing the definition of $\calC$, we have
\begin{equation*}
\begin{aligned}
a((1-\calC)v_H,(1-\calC)v_H)
&=a(v_H,v_H)-a(\calC v_H,\calC v_H)\leq\beta\|\nabla v_H\|_{\LL}^2.
\end{aligned}
\end{equation*}
With~\eqref{eq:CminCl}, we further estimate 
\begin{equation}
\begin{aligned}\label{eq:stab_1minCl}
a((1-\Cl)v_H, (1-\Cl)v_H)&=a((1-\calC)v_H, (1-\calC)v_H)+a(\CCl v_H, \CCl v_H)\\
&\leq \beta(1+\Cellsq)\|\nabla v_H\|_{\LL}^2.
\end{aligned}
\end{equation}
Using~\eqref{eq:Pi_stability}, \eqref{eq:Cb}, \eqref{eq:Cinv}, and \eqref{eq:stab_1minCl}, we can thus show that the last term in the discrete energy \eqref{eq:nrg_definition} is bounded by
\begin{equation}
\begin{aligned}\label{eq:stability_delta}
\tau^2(\tfrac14 -\theta)\, a\big(&\Dt\tul{\halfpos},\Dt\tul{\halfpos}\big)
\\&=\tau^2(\tfrac14 -\theta)\,a\big((1-\Cl)\calB_H\PiH \Dt\tul{\halfpos},(1-\Cl)\calB_H\PiH \Dt\tul{\halfpos}\big)\\
&\leq \tau^2(\tfrac14 -\theta)\,\beta(1+\Cellsq) \Con[b]^2\Con[inv]^2H^{-2}\big\|\Dt\tul{\halfpos}\big\|_{\LL}^2.
\end{aligned}
\end{equation}
Using the CFL condition~\eqref{eq:CFL} in estimate~\eqref{eq:stability_delta} thus yields the positivity of the discrete energy,
\begin{equation*}
\calE^{n+\frac{1}{2}}\geq\frac{1}{2}\,a\big(\tul{n+\frac{1}{2}},\tul{n+\frac{1}{2}}\big) +\frac{1}{2}\,c_\delta\big\|\Dt\tul{\halfpos}\big\|_{\LL}^2\geq 0,
\end{equation*}
where, $c_\delta = 1-(1-\delta)^2$. 
Next, we show the estimate~\eqref{eq:stability_estimate}. It holds
\begin{equation}\label{eq:stability_norm_estimate}
\big\| \Dt\tul{\halfpos}\big\|_{\LL}+\big\|\nabla\tul{n+\frac{1}{2}}\big\|_{\LL}\leq\sqrt{2}\max\Big\{\sqrt{\tfrac{2}{c_\delta}},\sqrt{\tfrac{2}{\alpha}}\Big\}\sqrt{\calE^{n+\frac{1}{2}}}.
\end{equation}
To bound the last factor $\sqrt{\calE^{n+\frac{1}{2}}}$ on the right-hand side, we use once again the energy property~\eqref{eq:nrg_estimate} and get
\begin{equation*}
\begin{aligned}
\calE^{n+\frac{1}{2}}-\calE^{n-\frac{1}{2}}&=\tfrac{1}{2}\tau \big(f^{n;\theta},\Dt\tuH{\halfpos}+\Dt\tuH{\halfneg}\big)_{\LL}\\
&\leq \tfrac{1}{\sqrt{2c_\delta}}\tau\big\| f^{n;\theta }\big\|_{\LL}\Big(\sqrt{\calE^{n+\frac{1}{2}}}+\sqrt{\calE^{n-\frac{1}{2}}}\Big),
\end{aligned}
\end{equation*}
which iteratively leads to
\begin{equation}\label{eq:stability_nrg_estimate}
\sqrt{\calE^{n+\frac{1}{2}}}\leq \sqrt{\calE^{\frac{1}{2}}}+\sum_{k=1}^{n}\tfrac{1}{\sqrt{2c_\delta}}\tau\big\| f^{k;\theta }\big\|_{\LL}.
\end{equation}
Lastly, we bound the initial energy $\sqrt{\calE^{\frac{1}{2}}}$ in \eqref{eq:stability_nrg_estimate}. Re-ordering the terms of the energy \eqref{eq:nrg_definition} leads to 
\begin{align*}
	\calE^{\frac{1}{2}}&=\tfrac{1}{2}\Big[\big(\Dt\tul{\frac{1}{2}},\Dt\tul{\frac{1}{2}}\big)_{\LL}+\Big(a\big(\tul{\frac{1}{2}},\tul{\frac{1}{2}}\big)-\tfrac{1}{4}\tau^2a\big(\Dt\tul{\frac{1}{2}},\Dt\tul{\frac{1}{2}}\big)\Big)\\
	&\qquad+\theta\tau^2a\big(\Dt\tul{\frac{1}{2}},\Dt\tul{\frac{1}{2}}\big)\Big]\\
	&=\tfrac{1}{2}\Big[\big(\Dt\tul{\frac{1}{2}},\Dt\tul{\frac{1}{2}}\big)_{\LL}+a\big(\tul{1},\tul{0}\big)+\theta\tau^2a\big(\Dt\tul{\frac{1}{2}},\Dt\tul{\frac{1}{2}}\big)\Big].
\end{align*}
This allows us to estimate the initial energy by
\begin{equation}
\begin{aligned}\label{eq:stability_init_nrg_estimate}
\sqrt{\calE^{\frac{1}{2}}}
&\leq\frac{1}{\sqrt{2}}\Big(\big\|\Dt\tul{\frac{1}{2}}\big\|_{\LL}+\sqrt{\beta  \|\tul{1}\|_{\LL}\|\tul{0}\|_{\LL}}
+\tau\sqrt{\theta\beta}\big\|\nabla \Dt\tul{\frac{1}{2}}\big\|_{\LL}\Big).\\
\end{aligned}
\end{equation}
Combining the estimate~\eqref{eq:stability_norm_estimate} with the bounds~\eqref{eq:stability_nrg_estimate} and~\eqref{eq:stability_init_nrg_estimate} yields the stability estimate~\eqref{eq:stability_estimate}, where
\begin{equation*}
\Con[s] =\max\Big\{\sqrt{\tfrac{2}{c_\delta}},\sqrt{\tfrac{2}{\alpha}}\Big\}\max\Big\{\tfrac{1}{\sqrt{c_\delta}},\sqrt{\beta}\Big\}. \qedhere
\end{equation*}
\end{proof}

\begin{remark}
Note that the constants $\Con[b]$ and $\Con[inv]$ depend on $p$, thus the CFL condition depends on $p$ as well. While the theory in~\cite{DonHM23} predicts a quadratic scaling with respect to $p$, the practical scaling of the constants appears to be much better.
\end{remark}

\subsection{Error estimate for the \splod~method}\label{subsec:error}
In this subsection, we present and prove a full error estimate for the \splod~method. 
First, we give the general assumptions to obtain the convergence results.

\begin{assumption}\label{ass:regularity}
Assume that
\begin{itemize}
\item[] (A0)\; $f\in C^{5}([0,T];H^{k}(\Omega))\;\;\text{for some}\;\; k\in\N_0$, 
\item[] (A1)\; $u(0)=u_0\in \V,\quad\partial_t u(0) = v_0\in \V$,
\item[] (A2)\; $\partial_t^mu(0)\coloneqq \partial_t^{m-2}f(0)+\ddiv(A\nabla (\partial_t^{m-2}u(0)))\in \V,\quad m=2,\dots, 5$,
\item[] (A3)\; $\partial_t^{6}u(0)\coloneqq\partial_t^{4}f(0)+\ddiv(A\nabla(\partial_t^{4} u(0)))\in \LL$,
\item[] (A4)\; There exists a constant $\Con[init]^{\,0}>0$ (independent of $\varepsilon$) such that
\begin{equation*}
	\sum_{m=0}^{5}\|\partial_t^mu(0)\|_{\V}+\|\partial_t^{6}u(0)\|_{\LL}\leq \Con[init]^{\,0}.
\end{equation*}
\end{itemize}
\end{assumption}

\begin{remark}\label{rem:regularity}
Assumption \ref{ass:regularity} is also referred to as \emph{well-prepared and compatible of order~$5$}; cf.~\cite[Def.~4.5]{AbdH17}. Under this assumption, we obtain
\begin{equation*}
\partial_{t}^{6}u\in C([0,T];\LL),\quad\partial_{t}^mu\in C([0,T];\V),\quad m=0,\dots,5,
\end{equation*}
based on the result in \cite[Ch.~3]{LioM72} applied to the variational formulation~\eqref{eq:var_form} and its temporal derivatives. Furthermore, we can bound the norms of the solution from above and denote with $\Con[data]^{\,0}>0$ a generic constant (independent of the fine scale $\varepsilon$) such that
\begin{equation*}
\sum_{m=0}^{5}\|\partial_t^mu\|_{C([0,T];\V)}+\|\partial_{t}^{6}u\|_{C([0,T];\LL)}\lesssim \|f\|_{C^{5}([0,T];H^k(\Omega))}+\Con[init]^{\,0}\leq\Con[data]^{\,0}
\end{equation*}
holds. Note that the assumption~(A4) in that context is mainly required to ensure that the corresponding norms do not scale negatively with the oscillation scale on which the coefficient varies. 

Assumption~\ref{ass:regularity} with $k=0$ is the most general to show the desired convergence properties of our method. In order to further improve the rates, additional assumptions are required, specifically that $\partial_t^mu\in C([0,T];H^{j+1}(\Omega))$ for some $j\geq 1$ and $m\in\N$. Again, the conditions can be fulfilled by appropriate assumptions on the initial data, the right-hand side, the domain, and the coefficient. As this is not the target regime of this paper, the corresponding assumptions are presented in Appendix~\ref{app:regularity}. In the respective setting, we have to bound the norms of the solution and its temporal derivatives in their respective norms by a modified factor $\Con[data]^{\,j}\varepsilon^{-j}$ that depends on the fine-scale $\varepsilon$ and an adjusted constant depending on the regularity parameter $j$ and the corresponding norms of the initial data and right-hand side. 
\end{remark}

\begin{theorem}[Error of the \splod~method]\label{thm:error}
Suppose Assumption \ref{ass:regularity} holds for some~$k\in\N_0$ and additionally let $u\in C^4([0,T];H^{j+1}(\Omega))$ for some $j\in\N_0$ (we note that the case $j=0$ is covered by Assumption~\ref{ass:regularity}). Let $r=\min\{k+1,j+2,p+2\}$ and~$\ell\gtrsim r|\log H|$. Further, let $\mathbf{\tilde{u}^\ell_H}=\{\tul{n}\}_{n=0}^N$ be the solution to the \splod~method defined in~\eqref{eq:splodth} with the initial conditions defined below in \eqref{eq:initial} and~$u$ the solution to~\eqref{eq:var_form}. For $\theta<\frac14$, let the CFL condition \eqref{eq:CFL} hold. Then with~$t_n=n\tau$, we have
\begin{equation}\label{eq:error}
\Big\|\Dt\tul{\halfpos}-\frac{u(t_{n+1})-u(t_n)}{\tau}\Big\|_{\LL}+\big\|\nabla\big(\tul{n+\frac{1}{2}}-u(t_{n+\frac{1}{2}})\big)\big\|_{\LL}\lesssim (H^r+\tau^s)\,\Con[data]^{\,j}\varepsilon^{-j}
\end{equation}
with $s=2$ if $\theta\neq\frac{1}{12}$ and $s=4$ if $\theta=\frac{1}{12}$. In particular, for $j=0$ the result is independent of $\varepsilon$.
\end{theorem}

Before we prove the theorem, we give some auxiliary definitions. First, we introduce the function~$\tilde{z}_H^\ell$ as the solution to the semi-discrete problem
\begin{equation}\label{eq:semi_discrete}
(\partial_t^2\tilde{z}_H^\ell(t),\tvl)_{\LL}+a(\tilde{z}_H^\ell(t),\tvl)=(f(t),\tvl)_{\LL}
\end{equation}
for all $\tvl\in\tUl$ and $t\in[0,T]$ with initial conditions $\tilde{z}^\ell_H(0)=(1-\Cl)\calB_H\PiH u_0$ and~$\partial_t\tilde{z}^\ell_H(0)=(1-\Cl)\calB_H\PiH v_0$. Analogously to above, we write $\tilde{z}_H=\tilde{z}_H^\infty$ for $\ell=\infty$.

\begin{remark}\label{rem:regularity_z}
Differentiating~\eqref{eq:semi_discrete} with respect to time and using the regularity of the initial conditions (which follows from Assumption~\ref{ass:regularity}), we obtain by standard ODE theory that $\partial_t^4\tilde{z}^\ell_H\in C^2([0,T],\tUl)$ and $\tilde{z}^\ell_H\in C^6([0,T],\tUl)$. Furthermore, by \cite[Ch.~3]{LioM72} we have 
\begin{equation}\label{eq:Cdata}
\begin{aligned}
\sum_{k=0}^{5}\|\tilde{z}_H^\ell\|_{{C^k([0,T];{\V})}}+\|\tilde{z}_H^\ell\|_{{C^6([0,T];{\LL})}}{\lesssim}~\Con[data]^{0}.
\end{aligned}
\end{equation}
\end{remark}

Next, we define the operator $\tPiH$ as the orthogonal projection onto the global multiscale space $\tUH$ with respect to the bilinear form $a$, i.e., for $v\in\V$ we define $\tPiH v$ by
\begin{equation}\label{eq:tPi}
	a(\tPiH v,\tvH)=a(v,\tvH)
\end{equation}
for all $\tvH\in\tUH$. Let $u$ be the solution to the variational formulation \eqref{eq:var_form} with regularity properties as stated in Remark \ref{rem:regularity} and $\tul{n}\in\tUl$ for $\ell\in\N$ and $\tuH{n}\in\tUH$ for $\ell=\infty$ be the solutions to the $p$-LOD-$\theta$ method defined in \eqref{eq:splodth} for $n=0, \dots, N$ with initial conditions defined analogously to~\cite{Kar11} by $\tul{0}=\mathfrak{u}_0$ and $\tul{1}$ is obtained from
\begin{multline}\label{eq:initial}
	(\tul{1}-\tul{0},\tvl)_{\LL}+\tau^2\theta a(\tul{1}-\tul{0},\tvl)=\\
	\begin{aligned} &\tau(\mathfrak{v}_0,\tvl)_{\LL}-\tfrac{\tau^2}{2}a(\mathfrak{u}_0,\tvl)_{\LL}-\tfrac{\tau^3}{12}a(\mathfrak{v}_0,\tvl)_{\LL}
	+\tfrac{\tau^2}{2}(f(0),\tvl)_{\LL}\\
	&+\tfrac{\tau^3}{6}(\partial_tf(0),\tvl)_{\LL}+\tfrac{\tau^4}{24}(\partial_t^2f(0),\tvl)_{\LL}
	\end{aligned}
\end{multline}
for all $\tvl\in\tUl$, where $\mathfrak{u}_0=(1-\Cl)\calB_H\PiH u_0$ and $\mathfrak{v}_0=(1-\Cl)\calB_H\PiH v_0$. We note that this initial condition is chosen such that it represents the fourth-order Taylor polynomial of the semi-discrete solution at the initial time. 

The following lemma provides an estimate for the time discretization error. We make use of the stability estimate \eqref{eq:stability_estimate} to derive an error estimate and motivate the choice for the initial conditions \eqref{eq:initial}. The ideas in the proof follow~\cite{Kar11} with some technicalities related to the $p$-LOD method.

\begin{lemma}[Temporal error]\label{lem:zeta}
Let $\{\tul{n}\}_{n=0}^N$ and $\tilde{z}_H^\ell$ be the solutions to~\eqref{eq:splodth} and~\eqref{eq:semi_discrete}, respectively, and let $\zeta^n=\tul{n}-\tilde{z}_H^\ell(t_n)$ with $t_n=n\tau$.Further, suppose that Assumption~\ref{ass:regularity} holds. Then there exists a constant $\Con[sd]>0$, dependent on $T$, $\theta$, $\delta$, $\alpha$ and $\beta$ such that
\begin{equation}\label{eq:zeta}
\begin{aligned}
&\big\|\Dt\zeta^{\halfpos}\big\|_{\LL}+\big\|\nabla\zeta^{n+\frac{1}{2}}\big\|_{\LL}
\leq\tau^s\,\Con[sd]\Con[data]^{\,0},
\end{aligned}
\end{equation}
where $s=2$ if $\theta\neq\frac{1}{12}$ and $s=4$ if $\theta=\frac{1}{12}$.
\end{lemma}

\begin{proof} The sequence $\{\zeta^n\}_{n=0}^N$ solves 
\begin{equation*}
\begin{aligned}
\big(\Dt^2\zeta^{n},\tvl\big)_{\LL}+a\big(\zeta^{n;\theta},\tvl\big)=&\big(\partial_t^2\tilde{z}_H^{n,\ell;\theta},\tvl\big)_{\LL}-\big(\Dt^2\tilde{z}_H^\ell(t_{n}),\tvl\big)_{\LL}
\end{aligned}
\end{equation*}
for all $\tvl\in\tUl$ and $n=1,\cdots, N-1$, where $\tilde{z}_H^{n,\ell;\theta}$ is once again defined as in~\eqref{eq:theta_difference}. Moreover,  $\tilde{z}_H^{n,\ell}=\tilde{z}_H^\ell(t_n)$.
With the stability estimate~\eqref{eq:stability_estimate}, we therefore obtain
\begin{equation}\label{eq:zeta_stability}
\begin{aligned}
\big\|\Dt\zeta^{\halfpos}&\big\|_{\LL}+\big\|\nabla \zeta^{n+\frac{1}{2}}\big\|_{\LL}\\
&\leq \Con[s]\Big(\big\|\Dt\zeta^{\frac{1}{2}}\big\|_{\LL}+\sqrt{\| \zeta^{1}\|_{\LL}\|\zeta^{0}\|_{\LL}}+\tau\sqrt{\theta}\big\|\nabla \Dt\zeta^{\frac{1}{2}}\big\|_{\LL}\\
&\qquad\qquad+\sum_{k=1}^n\tau\big\|\partial_{t}^2\tilde{z}_H^{k,\ell;\theta}-\Dt^2\tilde{z}_H^\ell(t_{k})\big\|_{\LL}\Big).
\end{aligned}
\end{equation}
For the sum in the last row of~\eqref{eq:zeta_stability}, we obtain with a Taylor expansion the estimate
\begin{align}\label{eq:zeta_sum}
\sum_{k=1}^n\tau\big\|\partial_{t}^2\tilde{z}_H^{k,\ell;\theta}-\Dt^2\tilde{z}_H^\ell(t_{k})\big\|_{\LL}\lesssim \tau^s\,T\, \|\tilde{z}^\ell_H\|_{C^{s+2}([0,T];\LL)},
\end{align}
where $s=2$ if $\theta\neq\frac{1}{12}$ and $s=4$ if $\theta=\frac{1}{12}$. 
Since $\zeta^0=0$, the second term on the right-hand side of~\eqref{eq:zeta_stability} vanishes. It remains to bound the two other terms on the right-hand side. 
Subtracting the equation
\begin{multline*}
	\big(\tilde{z}_H^\ell(t_1)-\tilde{z}_H^\ell(t_0),\tvl\big)_{\LL}+\tau^2\theta a\big(\tilde{z}_H^\ell(t_1)-\tilde{z}_H^\ell(t_0),\tvl\big)\\
	=\big(\tilde{z}_H^\ell(t_1)-\tilde{z}_H^\ell(t_0),\tvl\big)_{\LL}-\tau^2\theta\big(\partial_t^2(\tilde{z}_H^\ell(t_1)-\tilde{z}_H^\ell(t_0)),\tvl\big)_{\LL}+\tau^2\theta\big(f(t_1)-f(t_0), \tvl\big)_{\LL}
\end{multline*}
from \eqref{eq:initial} and using the Taylor expansions of $f$ and $\tilde{z}^\ell_H$, we obtain
\begin{multline*}
	\big(\zeta^1-\zeta^0,\tvl\big)_{\LL}+\theta\tau^2a(\zeta^1-\zeta^0,\tvl\big)_{\LL}\\
	\begin{aligned}
		&=(\theta-\tfrac{1}{12})\tau^3\big[\big(\partial_t^3\tzl{0},\tvl\big)_{\LL}-\big(\partial_tf(t_0),\tvl\big)_{\LL}\big]\\
		&\quad+\tfrac{1}{2}(\theta-\tfrac{1}{12})\tau^4\big[\big(\partial_t^4\tzl{0},\tvl\big)_{\LL}-\big(\partial_t^2f(t_0),\tvl\big)_{\LL}\big]+\mathcal{O}(\tau^5).
	\end{aligned}
\end{multline*}
Choosing the test function $\tvl=\frac{\zeta^1-\zeta^0}{\tau}$ and dividing by $\tau$, we estimate using Young's inequality
\begin{multline*}
\|\Dt\zeta^{\frac{1}{2}}\|^2_{\LL}+\tau^2\theta\|\nabla \Dt\zeta^{\frac{1}{2}}\|^2_{\LL}\\
\lesssim \tau^s \big(\|\tilde{z}^\ell_H\|_{C^5([0,T];\LL)}+\|f\|_{C^3([0,T];\LL)}\big)\big(\|D_{\tau}\zeta^1\|_{\LL}+\tau\sqrt{\theta}\|\nabla D_\tau\zeta^1\|_{\LL}\big),
\end{multline*}
where $s=2$ if $\theta\neq\frac{1}{12}$ and $s=4$ if $\theta=\frac{1}{12}$. 
This yields
\begin{equation*}
\|D_{\tau}\zeta^1\|_{\LL}+\tau\sqrt{\theta}\|\nabla D_{\tau}\zeta^1\|_{\LL}\lesssim 2\tau^s\big(\|\tilde{z}^\ell_H\|_{C^{s+1}([0,T];\LL)}+\|f\|_{C^{s-1}([0,T];\LL)}\big),
\end{equation*}
and combined with~\eqref{eq:zeta_stability} and~\eqref{eq:zeta_sum} we obtain
\begin{multline*}
\|D_{\tau}\zeta^{n+1}\|_{\LL}+\|\nabla\zeta^{n+\frac{1}{2}}\|_{\LL}\\
\leq \tau^s\Con[sd]\big(\|\tilde{z}^\ell_H\|_{C^{s+2}([0,T];\LL)}+\|\tilde{z}^\ell_H\|_{C^{s+1}([0,T];\LL)}+\|f\|_{C^{s-1}([0,T];\LL)}\big)
\end{multline*}
for some $\Con[sd] > 0$. 
Note that the constant $\Con[sd]$ scales linearly with the final time $T$. Finally, \eqref{eq:zeta}~follows with Remark~\ref{rem:regularity_z}.
\end{proof}

\begin{remark}[Initial conditions]
If $\theta\neq\frac{1}{12}$, the convergence rate in time in Lemma~\ref{lem:zeta} is only $s=2$ and the initial condition $\tul{1}$ do not need to be of fourth order. Therefore, they can be computed by the reduced equation
\begin{equation*}
\begin{aligned}
(\tul{1}-\tul{0},\tvl)+\tau^2\theta a(\tul{1}-\tul{0},\tvl)=&\ \tau(v^0,\tvl)-\tfrac{\tau^2}{2}a(u^0,\tvl)+\tfrac{\tau^2}{2}(f(0),\tvl),
\end{aligned}
\end{equation*}
where $\tul{0}$, $u^0$ and $v^0$ are chosen as above; cf.~\cite{Kar11}. Note that in this case we require less regularity assumptions on $\tilde{z}^\ell_H$ and $f$.
\end{remark}

The following lemma quantifies another ingredient of the total error, namely the localization error that investigates the difference of using the correction operator~$\Cl$ or the global operator~$\calC$. The proof follows ideas presented in~\cite{GeeM23} adjusted to the higher-order setting.

\begin{lemma}[Localization error]\label{lem:eta}
Let $\tilde{z}_H^\ell$ and $\tilde{z}_H$ be solutions to~\eqref{eq:semi_discrete} with localization parameters~$\ell\in\N$ and $\ell=\infty$, respectively, and let $\eta(t)=\tzH{}-\tilde{z}_H^\ell(t)$. Further, suppose that Assumption~\ref{ass:regularity} holds. Then there exists a constant $\Con[le]>0$ dependent on $T$, the polynomial degree~$p$, $\alpha$, and $\beta$ such that
\begin{equation}\label{eq:eta}
\big\|\partial_t\eta(t)\big\|_{\LL}+\big\|\nabla\eta(t)\big\|_{\LL}\leq  \Con[le]\Cell\Con[data]^0.
\end{equation}
\end{lemma}

\begin{proof}
Let $z_H(t_{})=\calB_H\PiH\tzH{}$ and  $z_H^\ell(t)=\calB_H\PiH\tilde{z}_H^\ell(t)$. 
Since we are comparing solutions in the different spaces $\tzl{}\in\tUl$ and $\tzH{}\in\tUH$, we first use that
\begin{equation}\label{eq:CsolClsol}
	\begin{aligned}
		\| \partial_t\tzH{} - \partial_t\tzl{} \| &= \|(1-\calC)\partial_t z_H-(1-\Cl)\partial_t z_H^\ell\|_{\LL}\\&\leq \|(1-\Cl)\partial_t(z_H-z_H^\ell)\|_{\LL}+\|(\calC-\Cl)\partial_tz_H\|_{\LL},\\
		\|\nabla( \tzH{} - \tzl{} )\| &= \|\nabla((1-\calC)z_H-(1-\Cl)z_H^\ell)\|_{\LL}\\&\leq \|\nabla(1-\Cl)(z_H-z_H^\ell)\|_{\LL}+\|\nabla(\calC-\Cl)z_H\|_{\LL}.
	\end{aligned}
\end{equation}
The last terms on the right-hand sides of both inequalities will be estimated with Lemma~\ref{lem:CminCl} (after applying the Poincar\'e inequality for the last term in the first inequality). For the remaining terms, we set $\eH{} = z_H(t_{})-z^\ell_H(t) \in U_H$ and
\begin{equation*} 
	\tel{}=(1-\Cl)\eH{}=(1-\Cl)(z_H(t_{})-z^\ell_H(t)).
\end{equation*} 
The first terms on the right-hand side of~\eqref{eq:CsolClsol} then read
\begin{equation}
	\begin{aligned}\label{eq:teHl}
		\|(1-\Cl)\partial_t(z_H(t_{})-z^\ell_H(t))\|_{\LL}&=\|\partial_t\tel{}\|_{\LL},\\
		\|\nabla(1-\Cl)(z_H(t_{})-z^\ell_H(t))\|_{\LL}&=\|\nabla \tel{}\|_{\LL}.
	\end{aligned}
\end{equation}
To bound these norms, we observe that the function $\eH{}$ solves 
\begin{multline}\label{eq:eta_localization}
	((1-\Cl)\partial_t^2\eH{},(1-\Cl)\vH)+a((1-\Cl)\eH{},(1-\Cl)\vH)\\
	\begin{aligned}
		&=-(f(t),(\calC-\Cl)\vH)
		+((1-\calC)\partial_t^2 z_H(t),(\calC-\Cl)\vH)\\
		&\quad +((\calC-\Cl)\partial_t^2  z_H(t),(1-\Cl)\vH)
		+a((\calC-\Cl)z_H(t),(\calC-\Cl)\vH)
	\end{aligned}
\end{multline}
for all $\vH\in U_H$. 
We define an appropriate energy by
\begin{equation*}
E(t)\coloneqq \frac{1}{2}\big[\|(1-\Cl)\partial_t \eH{}\|_{\LL}^2+a\big((1-\Cl)\eH{},(1-\Cl)\eH{}\big)\big]
\end{equation*}
for all $t\in[0,T]$.
We now choose $v_H=\partial_t \eH{}$ as a test function in equation \eqref{eq:eta_localization} and integrate in time. Using~\eqref{eq:Cb}, \eqref{eq:Cpr}, \eqref{eq:Pi_stability}, \eqref{eq:CminCl}, and the fact that $e_H^\ell(0)=0$, we obtain after taking the supremum over $t\in[0,T]$ 
\begin{equation}\label{eq:eta_nrg}
\begin{aligned}
\operatorname{sup}_{t\in[0, T]}E(t)&\leq \big[H\Con[pr]\Con[b]\big(\|f\|_{C([0,T],\LL)}+T\|\partial_tf\|_{C([0,T],\LL)}\big)\\
&\qquad +H\Con[pr]\Con[b]\big(\|\partial_t^2\tilde{z}_H\|_{C([0,T],\LL)}+T\|\partial_t^3\tilde{z}_H\|_{C([0,T],\LL)}\big)\\
&\qquad +H^2\Con[pr]^2\Con[b]\big(\|\nabla\partial_t^2\tilde{z}_H\|_{C([0,T],\LL)}+T\|\nabla\partial_t^3\tilde{z}_H\|_{C([0,T],\LL)}\big)\\
&\qquad +\beta\Con[b]^2\Cell\big(\|\nabla \tilde{z}_H\|_{C([0,T],\LL)}+T\|\nabla\partial_t\tilde{z}_H\|_{C([0,T],\LL)}\big)\big]\\
&\quad\cdot\Cell\sqrt{\tfrac{2}{\alpha}} \big(\operatorname{sup}_{t\in[0, T]}E(t)\big)^{\frac{1}{2}}.
\end{aligned}
\end{equation}
With \eqref{eq:eta_nrg}, we can now bound the terms in~\eqref{eq:teHl}. Going back to~\eqref{eq:CsolClsol}, we finally obtain~\eqref{eq:eta} with~Lemma~\ref{lem:CminCl} and Remark~\ref{rem:regularity_z}.
\end{proof}

\begin{remark}\label{rem:loc}
If we choose the localization parameter $\ell\gtrsim r|\log H|$, then the error estimate~\eqref{eq:eta} reduces to
\begin{equation*}
	\big\|\partial_t\eta(t)\big\|_{\LL}+\big\|\nabla\eta(t)\big\|_{\LL}\lesssim H^{r}\,\Con[data]^0,
\end{equation*}
which matches the spatial error estimate as investigated below.
\end{remark}

To bound the spatial error, we split the error into two parts. The first part describes the error between the semi-discrete solution and the projection of the exact solution into the multiscale space $\tUl$. The second part of the spatial error is the error between the exact solution and its projection. We bound the two terms in the following two lemmas. 

\begin{lemma}[Spatial discretization error]\label{lem:phi}
Let $\tilde{z}_H$ and $u$ be the solutions to~\eqref{eq:semi_discrete} and~\eqref{eq:var_form}, respectively, and let $\tilde{\calP}$ be the projection defined in~\eqref{eq:tPi}. Further, let $\varphi(t)=\tzH{}-\tPiH u(t)$ and~$\rho(t)=u(t)-\tPiH u(t)$. If Assumption~\ref{ass:regularity} is fulfilled, then there exists a constant $\Con[sdp]>0$ (dependent on~$\alpha$) such that
\begin{equation}\label{eq:phi}
\|\partial_t\varphi(t)\|_{\LL}+\|\nabla\varphi(t)\|_{\LL}\leq\Con[sdp]\int\limits_0^t\|\partial_t^2\rho(s)\|_{\LL}\ds .
\end{equation}
\end{lemma}

\begin{proof} The result follows from~\cite{Jol03}. For any $t\in[0,T]$, the function $\varphi(t)$ solves the semi-discrete problem
\begin{equation}\label{eq:phi_wave}
(\partial^2_t\varphi(t),\tvH)_{\LL}+a(\varphi(t),\tvH)=(\partial^2_t\rho(t),\tvH)_{\LL}.
\end{equation}
for all~$\tvH\in\tUH$.
We define the energy
\begin{equation*}
\frakE(t)=\tfrac{1}{2}\big[\|\partial_t\varphi(t)\|^2_{\LL}+a(\varphi(t),\varphi(t))\big],
\end{equation*}
choose $\tvH=\partial_t{\varphi}(t)$ as test function in~\eqref{eq:phi_wave}, and integrate in time from $0$ to $t$, we obtain
\begin{equation*}
\sqrt{\frakE(t)}\leq \sqrt{\frakE(0)}+\sqrt{2}\int\limits_0^t\|\partial^2_t\rho(s)\|_{ \LL}\ds.
\end{equation*}
Since~$\frakE(0)=0$, this is the assertion.
\end{proof}

\begin{lemma}[Projection error]\label{lem:rho}
Let $u$ be the solution to~\eqref{eq:var_form} and~$\tilde{\calP}$ the projection defined in~\eqref{eq:tPi}. Further, let $\rho(t)=u(t)-\tPiH u(t)$ and suppose that Assumption~\ref{ass:regularity} holds for some $k\in\N_0$. Further, suppose that $u\in C^{2+\nu}([0,T];H^{j+1}(\Omega))$ for some $j\in\N_0$ and $\nu=0,1,2$ and set $r=\min\{j+2,k+1,p+2\}$. Then there exists a constant~$\Con[p]>0$ dependent on the polynomial degree~$p$ and~$\alpha$ such that 
\begin{equation}
\begin{aligned}\label{eq:rho}
	\|\nabla\partial_t^\nu\rho(t)\|_{\LL}&\leq\Con[p]  H^{r}\big(\|\partial_t^\nu f\|_{C([0,T];H^{k}(\Omega))}+\|\partial_t^{2+\nu}{u}\|_{C([0,T];H^{j+1}(\Omega))}\big),\\
	\|\partial_t^\nu\rho(t)\|_{\LL}&\leq \Con[p]\Con[pr] H^{r+1}\big(\|\partial_t^\nu f\|_{C([0,T];H^{k}(\Omega))}+\|\partial_t^{2+\nu}{u}\|_{C([0,T];H^{j+1}(\Omega))}\big).
\end{aligned}
\end{equation}
\end{lemma}

\begin{proof} For any $t\in[0,T]$, we have that $\partial_t^\nu\rho(t)\in\ker\PiH$ and combined with~\eqref{eq:Cpr} we obtain
\begin{equation*}
	\|\partial_t^\nu\rho(t)\|_{\LL}=\|(1-\PiH)\partial_t^\nu\rho(t)\|_{\LL}\leq\Con[pr] H\|\nabla\partial_t^\nu\rho(t)\|_{\LL}.
\end{equation*}
Further, using the orthogonality property of $\PiH$ and~\eqref{eq:Cpr}, we have 
\begin{equation}\label{eq:rho_jk}
	\begin{aligned}
		\alpha\|\nabla\partial_t^\nu\rho(t)\|^2_{\LL}&\leq a\big(\partial_t^\nu\rho(t),\partial_t^\nu\rho(t)\big)=a\big(\partial_t^\nu u(t)-\tPiH \partial_t^\nu u(t),\partial_t^\nu u(t)-\tPiH \partial_t^\nu u(t)\big)\\
		&= a\big(\partial_t^\nu u(t),\partial_t^\nu u(t)-\tPiH \partial_t^\nu u(t)\big)\\
		&=\big(\partial_t^\nu f(t)-\partial_t^{2+\nu}u(t),\partial_t^\nu u(t)-\tPiH \partial_t^\nu u(t)\big)_{\LL}\\
		&= \big((1-\PiH)(\partial_t^\nu f(t)-\partial_t^{2+\nu}u(t)),(1-\Pi_H)(\partial_t^\nu u(t)-\tPiH \partial_t^\nu u(t))\big)_{\LL}\\
		&\leq\Con[pr]\big(H^{k}\|\partial_t^\nu f(t)\|_{H^{k}(\Omega)}+H^{j+1}\|\partial_t^{2+\nu}u(t)\|_{H^{j+1}(\Omega)}\big)\Con[pr] H\|\nabla\partial_t^\nu\rho(t)\|_{\LL}.
	\end{aligned}
\end{equation}
Combining both equations gives the desired result.
\end{proof}

As mentioned above, the error estimate for the full error of the \splod~method is split into four parts, which we have analyzed above and can now put together to obtain the main result.

\begin{proof}[Proof of Theorem \ref{thm:error}.]
Let $e^n=\tul{n}-u(t_n)$. We have
\begin{equation}\label{eq:error_split}
e^n=\zeta^n-\eta(t_n)+\varphi(t_n)-\rho(t_n),
\end{equation}
where $\zeta^n=\tul{n}-\tilde{z}_H^\ell(t_{n})$, $\eta(t)=\tilde{z}_H(t)-\tilde{z}_H^\ell(t)$, $\varphi(t_n)=\tzH{n}-\tPiH u(t_n)$ and $\rho(t)=u(t)-\tPiH u(t)$ as in the above lemmas. The full error of the $p$-LOD-$\theta$ method can thus be bounded by
\begin{equation}
\begin{aligned}\label{eq:error_sum}
\|\Dt e^{\halfpos}\|_{\LL}+\|\nabla e^{\halfpos}\|_{\LL}&\leq  \|\Dt \zeta^{\halfpos}\|_{\LL}+\|\nabla \zeta^{\halfpos}\|_{\LL}\\
&\quad+\|\Dt \eta(t_{\halfpos})\|_{\LL}+\|\nabla \eta(t_{\halfpos})\|_{\LL}\\
&\quad +\|\Dt \varphi(t_{\halfpos})\|_{\LL}+\|\nabla \varphi(t_{\halfpos})\|_{\LL}\\
&\quad +\|\Dt \rho(t_{\halfpos})\|_{\LL}+\|\nabla \rho(t_{\halfpos})\|_{\LL}.
\end{aligned}
\end{equation}
Using a Taylor expansion of $\Dt\varphi(t_{\halfpos})$, Lemma~\ref{lem:phi}, and the continuity in time, there exist a $\xi\in[0,t_{n+1}]$ with 
\begin{equation*}
\|\Dt \varphi(t_{\halfpos})\|_{\LL}+\|\nabla \varphi(t_{\halfpos})\|_{\LL}\leq2\Con[sdp]\int\limits_0^{t_{n+1}}\|\partial_t^2\rho(s)\|_{\LL}\leq2\Con[sdp] T\|\partial_t^2\rho(\xi)\|_{\LL}.
\end{equation*}
With a Taylor expansion of $\Dt\rho(t_{\halfpos})$ and applying Lemma~\ref{lem:rho} we get 
\begin{equation*}
\begin{aligned}
\|\nabla \rho(t_{\halfpos})\|_{\LL}&\lesssim H^r\big(\|f\|_{C([0,T];H^{k}(\Omega))}+\|u\|_{C^2([0,T];H^{j+1}(\Omega))}\big),\\
\|\Dt \rho(t_{\halfpos})\|_{\LL}&\lesssim H^{r+1}\big(\|f\|_{C^1([0,T];H^{k}(\Omega))}+\|u\|_{C^3([0,T];H^{j+1}(\Omega))}\big),\\
\|\partial_t^2 \rho(\xi)\|_{\LL}&\lesssim H^{r+1}\big(\|f\|_{C^2([0,T];H^{k}(\Omega))}+\|u\|_{C^4([0,T];H^{j+1}(\Omega))}\big).
\end{aligned}
\end{equation*}
Finally, with~\eqref{eq:error_sum}, Lemma~\ref{lem:zeta}, and Lemma~\ref{lem:eta} with $\ell\gtrsim r|\log H|$, we obtain using a Taylor expansion on $D_\tau\eta(t_{\halfpos})$
\begin{equation*}
\|\Dt e^{\halfpos}\|_{\LL}+\|\nabla e^{\halfpos}\|_{\LL}\lesssim (H^r+\tau^s)\,\Con[data]^{\,j}\varepsilon^{-j}.
\end{equation*}
Note that the hidden constant depends on the final time $T$, $\delta$, $\alpha$, $\beta$, and the polynomial degree $p$ and is independent of the mesh size $H$, the time step size $\tau$, and the scale $\varepsilon$ on which the coefficient varies. Note that $s=2$ if $\theta\neq\frac{1}{12}$ and~$s=4$ if $\theta=\frac{1}{12}$. 
\end{proof}

\begin{remark}\label{rem:jk}
If the source function $f$ is smooth enough in space and time, i.e. $k\geq 2$, the critical term with the lowest order in the full error estimate is the term $H^{j+2}\,|\partial_t^2u|_{H^{j+1}}$; cf.~the last term in equation \eqref{eq:error_sum} estimated with \eqref{eq:rho_jk}.
From Assumption~\ref{ass:regularity}, we know that $j\geq 0$ such that we get a convergence rate of at least order $r\geq 2$. In the case of rough~$L^\infty$-coefficients, we cannot expect more than $H^1$-regularity in space for the function~$\partial_t^2{u}$. That is, $r=2$ is the highest possible rate under minimal assumptions on the coefficient. This reduced order can be explained by the $p$-LOD method being designed for the elliptic case where we can obtain the higher order rates from the right-hand sides only. This property is not directly inherited for the wave equation as we can see in~\eqref{eq:rho_jk}, where we also need regularity assumptions on the solution to obtain even higher orders of convergence.

We note, however, that the constant in front of the $H^2$-term scales like~$\Con[pr]^2\lesssim (p+1)^{-2}$ such that we expect the error to be smaller even if the rate does not increase. This scaling is also observed in the numerical examples in Section~\ref{sec:examples}. 
If the coefficient~$A$ is smoother, we may obtain higher regularity of $\partial_t^2{u}$ and, thus, an increased convergence rate, which is then capped at $r= p+2$. Note that if the coefficient is smooth but still has multiscale features, then the higher rate generally cannot be observed, since the spatial derivatives scale negatively with $\varepsilon$ if the coefficient varies on the scale $\varepsilon$. In the next section, we present several numerical examples which indicate the reduced order in the general $L^\infty$-setting and a higher order convergence for smooth and slowly oscillating coefficients.
\end{remark}

\begin{remark}\label{rem:L2}
Analogously to \cite{Kar11}, we may obtain an $L^\infty(0,T;L^2(\Omega))$-error estimate that essentially reads
\begin{equation*}
\max_{0\leq n\leq N}\|\tuH{n}-u(t_n)\|_{\LL}\lesssim (H^{r+1}+\tau^s)\,\Con[data]^{\,j}\varepsilon^{-j}.
\end{equation*}
That is, compared to the $L^\infty(0,T;H^1(\Omega))$-error estimate shown above we obtain an additional order in space with $r\leq p+2$, dependent on the regularity of $f$ and $\partial_t^2u$ with analogue argumentation as in Remark~\ref{rem:jk}. Further, we have $s=2$ if $\theta\neq\frac{1}{12}$ and $s=4$ if~$\theta=\frac{1}{12}$.
\end{remark}

\section{Numerical examples}\label{sec:examples}
\begin{figure}
\centering
\scalebox{.7}{\includegraphics{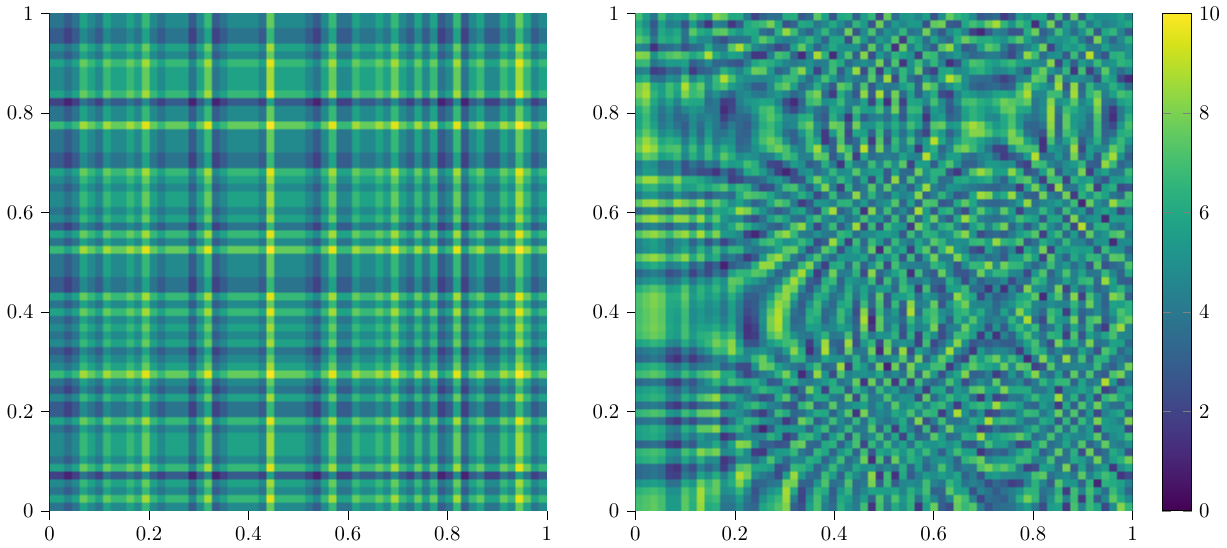}}
\caption{Coefficient $A_1$ (left) with values between $1$ and $10$ and coefficient $A_2$ (right) with values between $1$ and $9$. Both coefficients vary on the scale $\varepsilon=2^{-6}$. }
\label{fig:coefficients}
\end{figure}

In this section, we present several examples to verify the results presented in the previous sections. So far, we have assumed that the basis functions in $\tUl$ can be computed exactly. This is not possible in practice since the corrector problems~\eqref{eq:Cl_definition} are posed in infinite-dimensional spaces. Instead, we choose a finite element space $V_h\subset\V$ and replace the space~$\V$ by $V_h$ with mesh size $h\ll H$ in the above construction, which leads to error estimates with respect to a fine Galerkin finite element solution~$u_h$ instead of~$u$ with minor modifications in the above analysis. Roughly speaking, this also corresponds to a (fine enough) discretization $\Cl_h$ of the correction operator $\Cl$. More details can be found in~\cite[Sec.~4.3]{Mai21}. In the following, we therefore compute fully discrete {\splod}~solutions $\tilde{u}_{H,h}^{\ell,n}\in(1-\Cl_h)U_H\subset V_h$ and compare the results to a reference Galerkin finite element approximation $u_h\in V_h$ of~\eqref{eq:var_form}.

Throughout, we choose $\Omega=(0,1)^2$ with $T=1$ and the two multiscale coefficients $A_1$ and~$A_2$ shown in Figure~\ref{fig:coefficients}, which both oscillate on the scale $\varepsilon=2^{-6}$. The first coefficient~$A_1$ takes values between~$1$ and $10$ and the second coefficient $A_2$ varies between~$1$ and~$9$. The reference solution and the corrector problems are solved on meshes with mesh size $h=2^{-8}$, and we measure the errors at the final time $T=1$ in the norm~$|\cdot|_a = \sqrt{a(\cdot, \cdot)}$ if not stated otherwise. In the plots, we show errors with respect to the mesh size $H$ and the number of degrees of freedom $\calN$, respectively. 

\begin{figure}
	\centering
	\begin{subfigure}{.5\textwidth}
		\centering
		\scalebox{.6}{\includegraphics{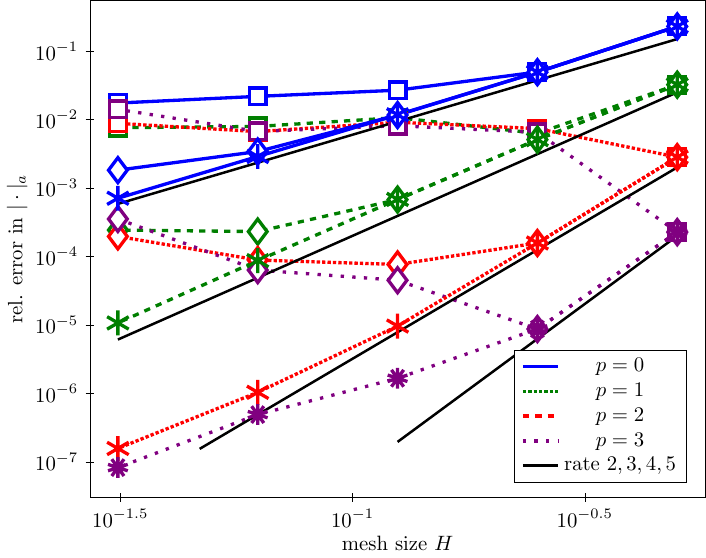}}
	\end{subfigure}%
	\begin{subfigure}{.5\textwidth}
		\centering
		\scalebox{.6}{\includegraphics{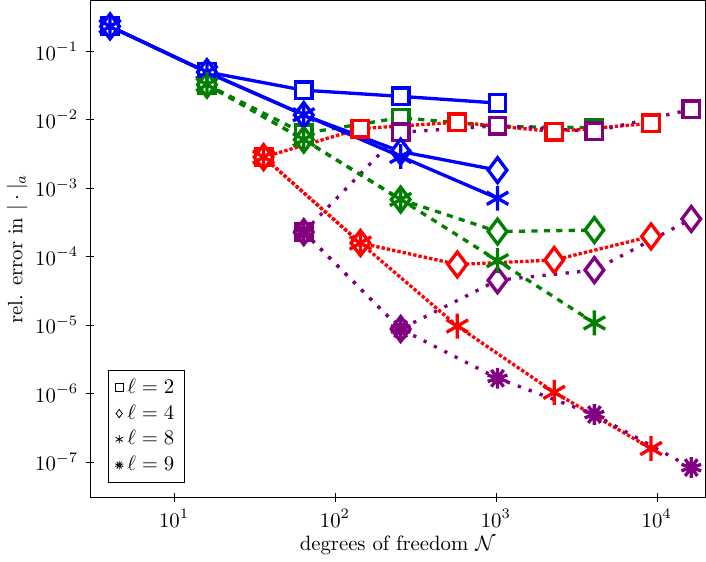}}
	\end{subfigure}
	\caption{Relative $|\cdot|_a$-errors for the coefficient $A_1$ with respect to the coarse mesh size $H$ (left) and the degrees of freedom $\calN$ (right) for different polynomial degrees $p$ and localization parameters~$\ell$.}
	\label{fig:A1}
\end{figure}

\subsection*{First example} 
For the first example, we prescribe homogeneous initial conditions $u_0=v_0=0$, a right-hand side~$f=\sin(\pi x_1)\sin(\pi x_2)\sin^4(t)$, and use the coefficient $A_1$ depicted in Figure~\ref{fig:coefficients} (left). We use the Crank--Nicolson scheme, i.e., $\theta=\frac{1}{4}$ for the time integration with a constant time step size $\tau=2^{-9}$ for all \splod~solutions as well as the reference solution to investigate the spatial errors only. The results are plotted in Figure~\ref{fig:A1}.
In the left plot, we observe a convergence rate of at least two with respect to the mesh size~$H$ for each of the polynomial degrees, which is generally in line with the error analysis in the previous section. Moreover, we observe that the errors for $p=1$ converge with order three, which is higher than expected, and even for $p=2$ and larger values of~$H$ a higher order rate can be observed. This may be explained by the fact that the term that leads to the (reduced) second order convergence is of lower magnitude such that its effect is only observed for smaller values of the error, where the reduced rate is clearly visible. We emphasize that the observed reduction of the convergence rate is not related to the localization error as it persists also when increasing~$\ell$. 
In the right plot, the error is depicted with respect to the number of degrees of freedom~$\calN$ and we observe that higher polynomial degrees overall lead to smaller errors such that moderately increasing the polynomial degree appears to be beneficial even if arbitrarily high orders of convergence cannot be expected. The reasoning for this observation is the scaling of the constants in the above error estimates, which scale like~$(p+1)^{-2}$ (cf. Remark \ref{rem:jk}). 
Both plots in Figure~\ref{fig:A1} show that for a small localization parameter~$\ell$, especially with higher polynomial degrees~$p$, the error stagnates with decreasing mesh size, which is related to the regimes where the localization error dominates. In order to increase visibility of further plots, we omit all lines with a small localization parameter and instead indicate with the marker which localization is needed for each polynomial degree and mesh size. This also makes sense as in practice one would choose the lowest possible localization that still yields the desired convergence behavior.

\begin{figure}
	\centering
	\begin{subfigure}{.5\textwidth}
		\centering
		\scalebox{.6}{\includegraphics{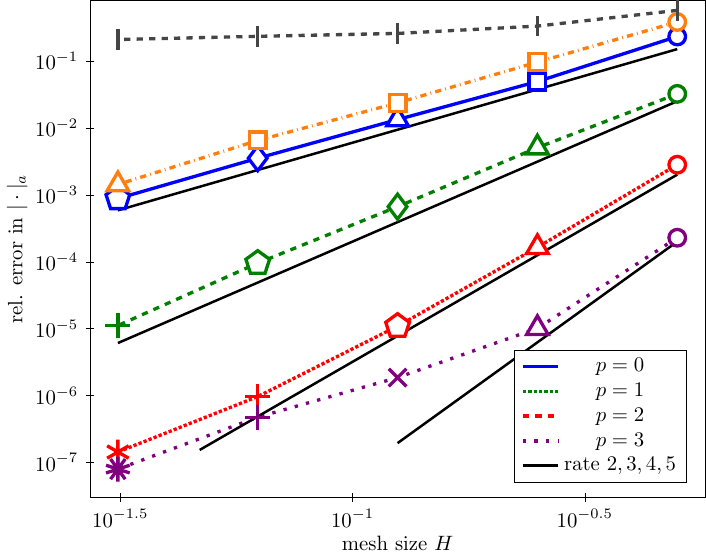}}
	\end{subfigure}%
	\begin{subfigure}{.5\textwidth}
		\centering
		\scalebox{.6}{\includegraphics{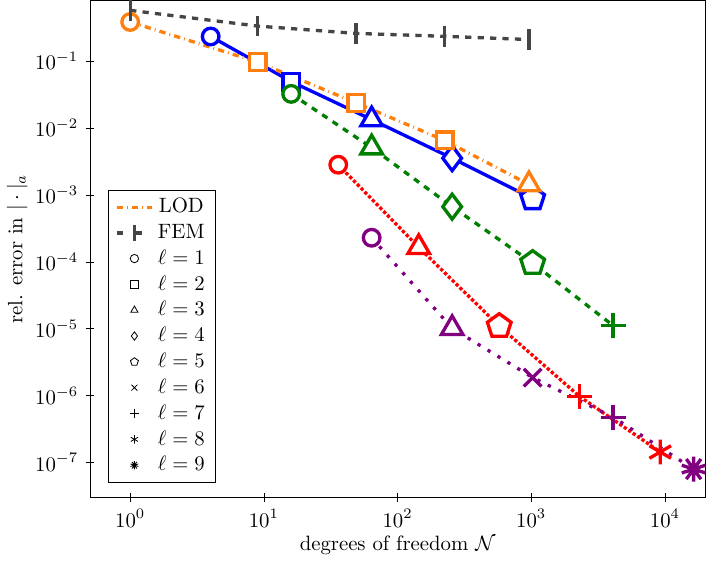}}
	\end{subfigure}
	\caption{Relative $|\cdot|_a$-errors for the second example with respect to the coarse mesh size $H$ (left) and the number of degrees of freedom $\calN$ (right) for different polynomial degrees $p$ and localization parameters~$\ell$.}
	\label{fig:A2}
\end{figure}

\subsection*{Second example} 

In Figure~\ref{fig:A2}, we show the errors of a second example based on the coefficient $A_2$ depicted in Figure~\ref{fig:coefficients} (right) as well as the initial conditions, the right-hand side, and the fine discretization parameters~$h$ and~$\tau$ as in the first example. In particular, we present the behavior of the \splod~method with~$\theta=\frac{1}{4}$ (Crank--Nicolson scheme) and compare it with other spatial discretization approaches, namely the classical FEM and the original first-order LOD method. We observe that the errors of the FEM are in the pre-asymptotic regime, where the mesh size~$H$ does not resolve the oscillation scale~$\varepsilon$ and the FEM solution fails to approximate the solution. The LOD method as used in~\cite{AbdH17} converges with order two, and the size of the errors are comparable to the errors of the \splod~method with polynomial degree $p=0$. From Figure~\ref{fig:A2} (right), we particularly observe that much lower errors with respect to the number of degrees of freedom compared to the original LOD method can be achieved by increasing the polynomial degree.
Analogously to above, we partially observe higher-order rates, which only for smaller magnitudes of the error are capped by two as predicted by the theory. 

\begin{figure}
	\centering
	\begin{subfigure}{.5\textwidth}
		\centering
		\scalebox{.6}{\includegraphics{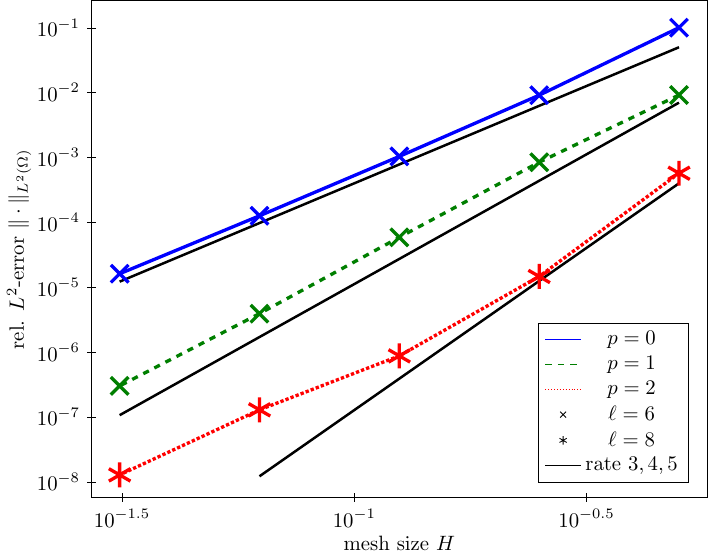}}
	\end{subfigure}%
	\begin{subfigure}{.5\textwidth}
		\centering
		\scalebox{.6}{\includegraphics{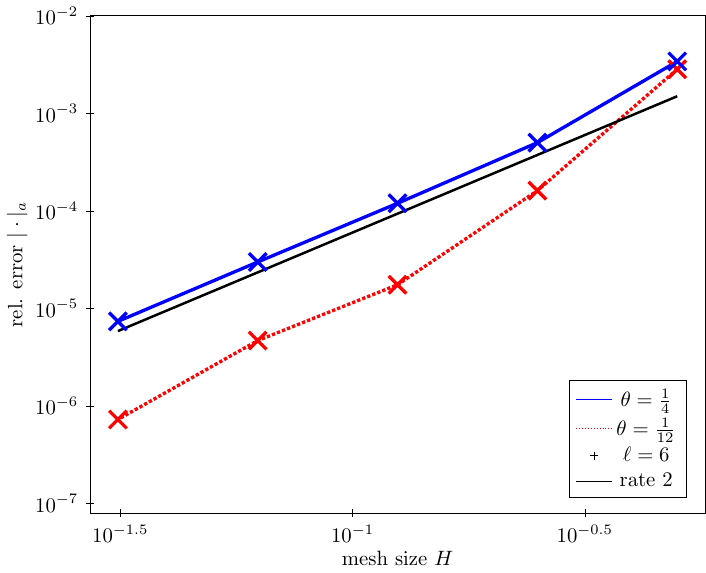}}
	\end{subfigure}
	\caption{Relative $L^2$-errors for the second example for $\theta=\frac{1}{12}$ (left) and relative $|\cdot|_a$-errors for the second example for $\theta=\frac{1}{4}$ compared to $\theta=\frac{1}{12}$ for $p=2$ (right) with time step sizes dependent on $H$ and $h$, respectively.}
	\label{fig:th}
\end{figure}

Since the convergence rates of the \splod~method are capped at $r=2$ in the most general $L^\infty$-setting, a general question is whether a higher-order time discretization is useful in the first place. An illustration is given in Figure~\ref{fig:th}. The plot on the left shows the relative errors in the $L^2$-norm in space. Here we choose~$\theta=\frac{1}{12}$ with coarse time step sizes~$\mathfrak{T}=2^{-5}H$ and fine time step sizes $\tau=2^{-5}h$ for the reference solution such that the CFL condition~\eqref{eq:CFL} holds. We observe convergence rates of one additional order compared to the errors measured in~$|\cdot|_a$, see Remark~\ref{rem:L2} and Figure~\ref{fig:A2}. The convergence rate for the Crank--Nicolson scheme in this example would be capped at $s=2$ and the higher-order would not be observed. 
Figure~\ref{fig:th}~(right) gives another reason to why the choice~$\theta=\frac{1}{12}$ can be better suited compared to other choices of~$\theta$. We choose both~$\theta = \frac14$ and~$\theta=\frac{1}{12}$ and time step sizes that depend on the mesh size, i.e., $\mathfrak{T}=2^{-5}H$ and $\tau=2^{-5}h$. Note that this choice of time steps is only required for the highest polynomial degree and could be reduced for smaller $p$. We observe that for the Crank--Nicolson scheme the relative $|\cdot|_a$-errors eventually hit the temporal error plateau. For the scheme with~$\theta=\frac{1}{12}$ we do obtain smaller errors since the temporal error plateau is not reached. 

\begin{figure}
	\centering
	\begin{subfigure}{.5\textwidth}
		\centering
		\scalebox{.6}{\includegraphics{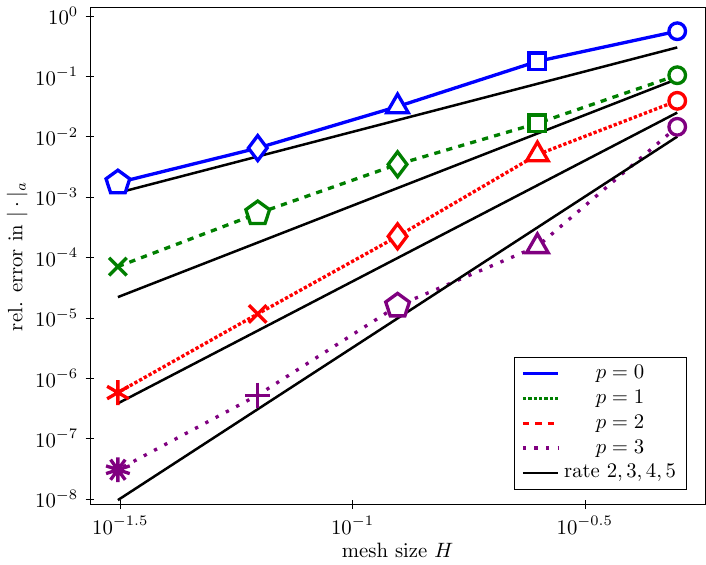}}
	\end{subfigure}%
	\begin{subfigure}{.5\textwidth}
		\centering
		\scalebox{.6}{\includegraphics{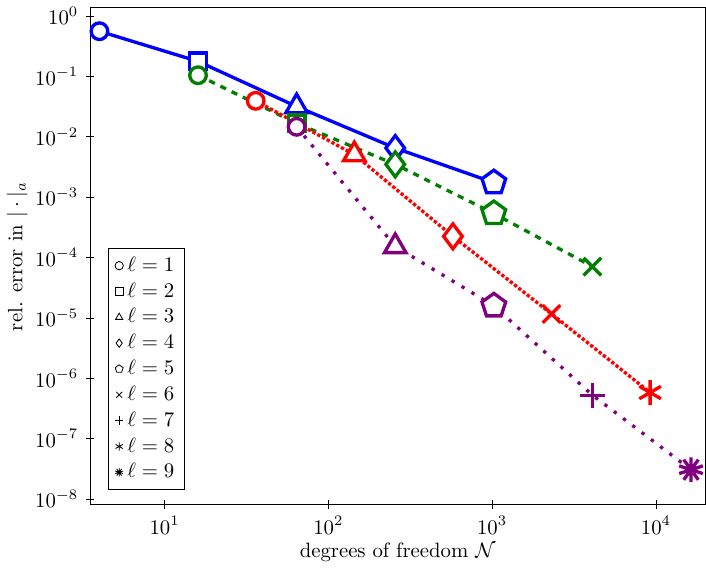}}
	\end{subfigure}
	\caption{Relative $|\cdot|_a$-errors for the third example with smooth coefficient and right-hand side for $\theta=\frac{1}{12}$ with respect to the coarse mesh size $H$ (left) and the number of degrees of freedom $\calN$ (right) for different polynomial degrees $p$ and localization parameters~$\ell$.}
	\label{fig:smooth}
\end{figure}

\subsection*{Third example}
In Figure \ref{fig:smooth}, we show the errors of a third example involving a smooth coefficient $A_3(x)=1+\frac{1}{2}\sin(x_1)\sin(2x_2)$, the right-hand side $f=\sin^4(\pi x_1)\sin^4(\pi x_2)\sin^4(t)$, and homogeneous initial conditions $u_0=v_0=0$. The depicted results are generally in line with Remark~\ref{rem:jk}. 
More precisely, in this example we have additional regularity of~$u$ in space and time, in particular the scaling of $H^{j}$-norms of $\partial_t^2u$ for some $j\in\N_0$ is not severe, since the coefficient only oscillates on a rather coarse scale. Based on Remark~\ref{rem:jk}, convergence rates in space of order $r=p+2$ can be expected. The plots in Figure~\ref{fig:smooth} show the errors of the \splod~method with $\theta=\frac{1}{12}$, $\mathfrak{T}=2^{-5}H$ and $\tau=2^{-5}h$, and we observe the additional rates as expected from the theory, which are capped at $4$ due to the temporal error.

\section{Conclusion}
In this paper, we have presented a method to solve the acoustic wave equation with oscillating coefficients, where we combined a higher-order extension of the LOD method for the spatial discretization with the $\theta$-scheme for time stepping. We have given rigorous a-priori error estimates and observed that arbitrarily high orders of convergence in space as in the elliptic case cannot be expected for the wave equation if only minimal assumptions on the coefficient hold. 
Nevertheless, our method can still achieve smaller errors by increasing the polynomial degree compared to lower-order approaches with a similar amount of degrees of freedom. This justifies a suitable increase of the polynomial degree. Further, if the coefficient is sufficiently smooth, higher-orders can be achieved.

Future research aims at modifying the presented approach to be better suited for an application to the heterogeneous wave equation. In particular, we aim at obtaining a truly higher-order method in space and time already under minimal assumptions on the coefficient.


\section*{Acknowledgments}
We would like to thank Moritz Hauck for fruitful discussions regarding the implementation of the method.

\section*{Funding}
Funding from the Deutsche Forschungsgemeinschaft (DFG, German Research Foundation) -- Project-ID 258734477 -- SFB 1173 is gratefully acknowledged.

\appendix

\section{Additional regularity}\label{app:regularity}
In this section, we provide the sufficient assumptions to obtain the higher regularity which is needed in Theorem~\ref{thm:error} to enable higher-order convergence rates. Note that we need the coefficient $A$ and also the boundary $\partial\Omega$ to be smooth enough and $A$ not to have multiscale features, i.e., the coefficient may only oscillate on a rather coarse scale $\varepsilon$. This is due to the norms of the spatial derivatives of the solution scaling negatively with $\varepsilon$. We fix $k,j\in\N$ with $k\geq j+2$ and introduce a stricter version of Assumption~\ref{ass:regularity}.
\begin{assumption}\label{ass:advanced}
	Assume that
	\begin{itemize}
		\item[(B0)]  $f\in C^{5}([0,T];H^k(\Omega)),\quad u(0)=u_0\in H^{j+3}(\Omega), \quad \partial_tu(0)=v_0\in H^{j+2}(\Omega)$,
		\item[(B1)]  $u_0\in \V,\quad v_0\in \V$,
		\item[(B2)]  $\partial_t^mu(0)\coloneqq \partial_t^{m-2}f(0)+\ddiv(A\nabla (\partial_t^{m-2}u(0)))\in \V$, for $m=2,\dots, \max\{5, j+2\}$,
		\item[(B3)]  $\partial_t^{\mathfrak{n}}u(0)\coloneqq\partial_t^{\mathfrak{n}-2}f(0)+\ddiv(A\nabla(\partial_t^{\mathfrak{n}-2} u(0)))\in \LL$, for $\mathfrak{n}=\max\{6, j+3\}$,
		\item[(B4)]  there exists a constant $\Con[init]^{j}>0$ (independent of $\varepsilon$) such that
		\begin{equation*}
			\|u_0\|_{H^{j+3}(\Omega)}+\|v_0\|_{H^{j+2}(\Omega)}+\sum_{m=2}^{\max\{5, j+2\}}\|\partial_t^mu(0)\|_{\V}+\|\partial_t^{\mathfrak{n}}u(0)\|_{\LL}\leq \Con[init]^{j}.
		\end{equation*}
	\end{itemize}
\end{assumption}

With similar arguments as in~\cite[Sec.~7.2]{Eva10} we then obtain under sufficient smoothness assumptions on $A$ and $\partial \Omega$ in addition to the regularity in Remark \ref{rem:regularity} that 
\begin{equation*}
	u\in C^m([0,T];H^{j+3-m}(\Omega)),\qquad m=0,\dots,j+3.
\end{equation*}
We can bound the norms of $u$ by the right-hand side and initial conditions as follows. Let~$k\geq j+2$ and the Assumption~\ref{ass:advanced} be fulfilled. Then, there exists a generic constant $\Con[data]^{j}>0$ that depends on the regularity parameter $j$ and is independent of the oscillation scale $\eps$, such that 
\begin{equation*}
	\begin{aligned}
		\sum_{m=0}^{3}\|\partial_t^mu\|_{C([0,T];H^{j+3-m}(\Omega))}&+\sum_{m=4}^{5}\|\partial_t^mu\|_{C([0,T];H^{1}_0(\Omega))}+\|\partial_t^6u\|_{C([0,T];\LL)}\\
		&\lesssim \varepsilon^{-j} \left[\|f\|_{C^{5}([0,T];H^{j+2}(\Omega))}+\Con[init]^{j}\right]
		\leq\Con[data]^{j}\varepsilon^{-j},\qquad j=1,\\
		\sum_{m=0}^{4}\|\partial_t^mu\|_{C([0,T];H^{j+3-m}(\Omega))}&+\|\partial_t^5u\|_{C([0,T];H^{1}_0(\Omega))}+\|\partial_t^6u\|_{C([0,T];\LL)}\\
		&\lesssim \varepsilon^{-j} \left[\|f\|_{C^{5}([0,T];H^{j+2}(\Omega))}+\Con[init]^{j}\right]
		\leq\Con[data]^{j}\varepsilon^{-j},\qquad j=2,\\
		\sum_{m=0}^{j+3}\|\partial_t^mu\|_{C([0,T];H^{j+3-m}(\Omega))}&\lesssim \varepsilon^{-j} \left[\|f\|_{C^{5}([0,T];H^{j+2}(\Omega))}+\Con[init]^{j}\right]\leq\Con[data]^{j}\varepsilon^{-j},\qquad j>2.
	\end{aligned}
\end{equation*}

\end{document}